\documentclass[11pt]{article}
\usepackage {amssymb} 
\usepackage {amsmath}
\usepackage{amsthm}
\usepackage{graphicx}
\usepackage {amscd}
\usepackage{subfigure}
\usepackage{color}
\usepackage{multirow}
\usepackage{tabu}
\usepackage[titletoc, title]{appendix}
\usepackage{diagbox}
\usepackage[colorlinks, citecolor=red]{hyperref}
\usepackage{stmaryrd}
\usepackage[all,color]{xy}

\usepackage{geometry}  

\newcommand{\vol}{{\rm vol}}
\newcommand{\ord}{{\rm ord}}
\newcommand{\fm}{\mathfrak{m}}
\newcommand{\fa}{\mathfrak{a}}
\newcommand{\cO}{\mathcal{O}}
\newcommand{\bR}{\mathbb{R}}
\newcommand{\bC}{\mathbb{C}}
\newcommand{\bZ}{\mathbb{Z}}
\newcommand{\lct}{{\rm lct}}
\newcommand{\wt}{{\rm wt}}

\newcommand{\Val}{{\rm Val}}

\newcommand{\hvol}{{\widehat{\rm vol}}}

\newcommand{\cF}{{\mathcal{F}}}
\newcommand{\bV}{{\mathbb{V}}}
\newcommand{\cI}{{\mathcal{I}}}

\newcommand{\cJ}{{\mathcal{J}}}
\newcommand{\fb}{{\mathfrak{b}}}

\newcommand{\bQ}{{\mathbb{Q}}}

\newcommand{\cX}{{\mathcal{X}}}
\newcommand{\cV}{{\mathcal{V}}}
\newcommand{\cL}{{\mathcal{L}}}

\newcommand{\cR}{{\mathcal{R}}}
\newcommand{\cY}{{\mathcal{Y}}}

\newcommand{\cW}{{\mathcal{W}}}

\newcommand{\bG}{{\mathbb{G}}}

\newcommand{\bP}{{\mathbb{P}}}

\newcommand{\fc}{{\mathfrak{c}}}

\newcommand{\bin}{{\bf in}}
\newcommand{\cC}{{\mathcal{C}}}
\newcommand{\cS}{{\mathcal{S}}}

\newcommand{\cA}{{\mathcal{A}}}

\newcommand{\cD}{{\mathcal{D}}}

\newcommand{\cH}{{\mathcal{H}}}

\newcommand{\cE}{{\mathcal{E}}}

\newcommand{\Spec}{\mathrm{Spec}}

\newcommand{\mult}{{\rm mult}}
\newcommand{\bA}{{\mathbb{A}}}
\newcommand{\cZ}{{\mathcal{Z}}}

\newcommand{\kB}{{\mathfrak{B}}}

\newcommand{\gr}{{\rm gr}}

\newtheorem{thm}{Theorem}[section]

\newtheorem{lem}[thm]{Lemma}

\newtheorem{defn}[thm]{Definition}

\newtheorem{prop}[thm]{Proposition}

\newtheorem{rem}[thm]{Remark}
\newtheorem{exmp}[thm]{Example}
\newtheorem{thmx}{Theorem}

\begin{document}

\title{ Stability of Valuations and Koll\'ar Components}
\author{Chi Li and Chenyang Xu}
\date{\today}
\maketitle
\abstract{We prove that among all Koll\'ar components obtained by plt blow ups of a klt singularity $o\in (X,D)$, there is at most one that is (log-)K-semistable. We achieve this by showing that if such a Koll\'ar component exists, it uniquely minimizes the normalized volume function introduced in \cite{Li15a} among all divisorial valuations. Conversely, we show that any divisorial minimizer of the normalized volume function yields a K-semistable Koll\'ar component.   We also prove that for any klt singularity, the infimum of the normalized volume function is always approximated by the normalized volumes of Koll\'ar components.} 
\tableofcontents

\section{Introduction}
Throughout this paper, we work over the field $\bC$ of complex numbers. It has been well known  by people working in higher dimensional geometry that there is an analogue between the local objects Kawamata log terminal (klt)  singularities, and its global counterparts  log Fano varieties (cf. e.g. \cite{Sho00, Xu14} etc.). From this comparison, since the stability theory of Fano varieties has been a central object in people's study in the last three decades, it is natural to expect that there is a local stability theory on singularities.  The primary goal of this preprint is to develope such a theory.  In another word,  we want to investigate singularities using the tools from the theory of K-stability, a notion which was first defined in \cite{Tia97} and  later algebraically formulated in \cite{Don02}.  We note that this interaction between birational geometry and K-stability theory  has been proved to significantly fertilize  both subjects (cf. \cite{Oda12, Oda13, LX14, WX14, LWX14, Fuj15} etc.).   

For the stability theory of log Fano varieties, a crucial ingredient is the CM weight. Philosophically, the stability of log Fano varieties is equivalent to minimizing the CM weight.  In the stability theory of singularities, we fix the singularity $(X,o)$ and look for `the most stable' valuation $v\in \Val_{X,o}$ which is centered over $o$. Thus the first step of establishing a local stability theory for $(X,o)$ would be to find the right counterpart of CM weight in the local setting. As a candidate the first named author defined in \cite{Li15a} the normalized volume function $\hvol_{(X,D),o}$ on the space of valuations centered at $o$. Its derivatives at the canonical divisorial valuation over a klt cone singularity along certain tangent directions associated to special test configurations are indeed the CM weights. So in some sense, using the local picture, the normalized volume function carries more information than the CM weight. 

By the above discussion and inspired by the global theory, we focus on studying the valuation that minimizes the normalized volume function, which is conjectured to uniquely exist and  ought to be thought as the `(semi-)stable' object.  This picture is understood well in the case of Sasakian geometry where one only considers the valuations coming from the Reeb vector fields induced by a good torus action (e.g. \cite{MSY08, CoSz16}). Here we can naturally compare the stability of the singularity with the stability for the base. However, this requires the extra cone structure. By investigating the minimizer of the normalized volume function on all valuations, our plan, as we mentioned, built on the previous work (\cite{Li15a, Li15b, LL16}), is to establish an intrinsic stability which only depends on the isomorphic class of the singularity.
 We recall that it was shown in \cite{Li15b, LL16} that a Fano manifold $X$ is K-semistable, if and only if that among all valuations over the vertex $o$ of the cone $C(X)$ given by a multiple of $-K_X$,  the canonical valuation obtained by blowing up the vertex $o\in C(X)$ minimizes the normalized volume function.  This gives evidence to justify that at least for these singularities, our study is in the right direction. 

For an {\it arbitrary} klt singularity, there is no direct way to associate a global object. Nevertheless, in differential geometry, when there is a `canonical' metric, the metric tangent cone around the singularity is the stable object in the category of metric spaces. With a similar philosophy,  we expect that the minimizer of the normalized volume function always gives {\it a degeneration} to a K-semistable Fano cone singularity in the Sasakian setting, and conversely any such degeneration should be provided by a minimizer of ${\hvol}_{X,o}$.  In the current paper, we work out this picture in the case that the minimizer is divisorial, by  implementing the machinery of the minimal model program (MMP) (based on the foundational results in \cite{BCHM10}). So our treatment will be purely algebraic though it is strongly inspired by analytic results in the study of K\"{a}hler-Einstein/Sasaki-Einstein metrics.

One ingredient we introduce is to define the volume associated to a birational model and then connect it to the normalized volume of a valuation. For studying the divisorial valuations, the class of models which play a central role here are the ones obtained in the construction of {\it Koll\'ar} component (cf. \cite{Xu14}): for an arbitrary $n$-dimensional klt singularity $(X,o)$, we can use minimal model program to construct a birational model whose exceptional locus is an $(n-1)$-dimensional log Fano variety. In this paper, we will systematically develop the tools of using Koll\'ar components to understand the normalized volume function and its minimizers. In fact, Koll\'ar components can be considered as the local analogue of special degenerations studied in \cite{LX14}. 
We also observe that in the set-up of Sasakian geometry, a Reeb vector gives rise to a Koll\'ar component if and only if it is rational, i.e., it is quasi-regular. 
  
Therefore, to summarize, the aim of this paper is of twofolds. On one hand, we aim to use the construction of Koll\'ar components to get information of the space of valuations, especially for the minimizers of the normalized volume function. On the other hand, in the reverse direction, we want to use the viewpoint of stability to study the construction of Koll\'ar components in birational geometry, and search out a more canonical object under suitable assumptions.   

We also expect that for any klt singularity $(X,o)$, even when the minimizer is not necessarily divisorial, we can still use suitable birational models to degenerate $(X,o)$ to a K-semistable (possibly irregular) singularity with a torus action of higher rank. However, it seems to involve a significant amount of new technical issues. 

In the below, we will give more details. 

\subsection{Koll\'ar components}
\begin{defn}[Koll\'ar component]\label{d-kollar}
Let $o\in (X,D)$ be a klt singularity. We call a proper birational morphism $\mu:Y\to X$ provides a {\bf Koll\'ar component $S$}, if $\mu$ is isomorphic over $X\setminus \{o\}$, and $\mu^{-1}(o)$ is an irreducible divisor $S$, such that $(Y,S+\mu^{-1}_*D)$ is purely log terminal (plt) and $-S$ is $\mathbb{Q}$-Cartier and ample over $X$.
\end{defn}

We easily see that the birational model $Y$ is uniquely determined once the divisorial valuation $S$ is fixed, and if we denote 
\begin{equation}\label{eq-diffS}
(K_Y+S+\mu_*^{-1}D)|_S=K_S+\Delta_S,
\end{equation}
(see \cite[Definition 4.2]{Kol13}), then $(S,\Delta_S)$ is a klt log Fano pair. 

 Given any klt singularity $o\in (X,D)$, after the necessary minimal model program type result is established (see \cite{BCHM10}), we know that there always exists a Koll\'ar component  (see \cite{Pro00} or \cite[Lemma 1]{Xu14}), but  it is often not unique (nevertheless, see the discussion in \ref{e-example}.4 for some known special cases for the uniqueness). From what we have discussed, instead of an arbitrary Koll\'ar component, we want to study those which are `the most stable', and show that they yield canonical objects if exist. Indeed, we shall prove that if there is a K-semistable Koll\'ar component, then it gives the unique minimizer of $\hvol_{(X,D),o}$ among all Koll\'ar components (actually even among all divisorial valuations).   
 
 \vskip 1mm
\vskip 1mm 
 
Compared to the global theory of degeneration of Fano varieties, this fits into the philosophy that K-stability provides a canonical degeneration (cf. \cite{LWX14, SSY14}) and it should minimize the CM weight among all degenerations.
 
 \bigskip
 The following theorem is our main result. See Definition \ref{def-hvol} for the definition of $\hvol_{(X,D),o}$.
 \begin{thm}\label{t-main1}Let $o\in (X,D)$ be a klt singularity.  
A divisorial valuation 
$\ord_S$ is a minimizer of $\hvol_{(X,D),o}$ if and only if the following conditions are satisfied
\begin{enumerate}
\item
$S$ is a Koll\'{a}r component; 
\item
$(S, \Delta_S)$ is K-semistable.
\end{enumerate}
Moreover, such a minimizing divisorial valuation, if exists, is unique among all divisorial valuations.
\end{thm}
 
We do not know, up to a  rescaling, whether such a valuation as in Theorem \ref{t-main1} is the unique minimizer of $\hvol_{(X,D),o}$ among {\it all} valuations in $\Val_{X,o}$ (see \cite{LX17} for further results).

\bigskip 
 
More concretely, we will prove Theorem \ref{t-main1} by establishing the following four theorems. We will need different techniques to prove each of them. 

First we prove
\begin{thmx}\label{t-main}
Let $o\in (X,D)$ be an algebraic klt singularity. Let $S$ be a Koll\'ar component over $X$.
If $(S,\Delta_S)$ is (log-)$K$-semistable. Then $\hvol_{(X,D),o}$ is minimized at the valuation $\ord_S$. 

\end{thmx}
 
This extends the main theorem in \cite{LL16} from cone singularities to the more general setting. For the proof, we need to degenerate a general klt singularity to a cone singularity induced by its Koll\'ar components. However, instead of degenerating the valuation, we degenerate the associated valuative ideals. We will also use a result in \cite{Liu16} which computes the infimum of normalized volumes using some normalized multiplicities. The latter was first considered in the work of de Fernex-Ein-Musta\c{t}\u{a} \cite{dFEM04} and its behavior under degeneration of singularities can be studied as in \cite{Mus02}. 

An extra subtlety is that we can not directly use \cite{LL16} since the result there was proved for the cone singularity over a $\mathbb{Q}$-Fano variety that specially degenerates to a K\"{a}hler-Einstein $\bQ$-Fano variety. It is conjectured that any K-semistable $\bQ$-Fano variety has a such degeneration. Here we can indeed circumvent this difficulty in two different ways. In one way, we will first show that it suffices to concentrate on the torus equivariant data (see Section \ref{s-equiv}) and then use a similar argument as in \cite{Li15b} to complete the proof. In an alternative way, we solve the question proposed in \cite{LL16} and hence can use the strategy there to prove the version we need (see Proposition \ref{p-fanocone}).

In  Section \ref{s-exam}, we use this criterion to find minimizers for various examples of singularities including: quotient singularities, $A_k$ and $E_k$ singularities etc.

\bigskip

Next, we turn to the result on the uniqueness. 

\begin{thmx}\label{t-main2}
If $o\in (X,D)$ is an algebraic klt singularity. Assume $S$ is a Koll\'ar component over $X$ such that $(S,\Delta_S)$ is $K$-semistable. Then 
$$\hvol_{(X,D),o}(\ord_S)<\hvol_{(X,D),o}(\ord_{S'})$$ for any other divisorial valuation $S'$. 
\end{thmx}

This is done by a detailed study of the geometry when the equality holds. In the cone singularity case, we investigate when the equality holds  in the calculation in \cite{LL16}. It posts a strong condition which enables us to compute the corresponding invariants including nef thresholds and pseudo-effective thresholds. The argument is partially inspired by the work in \cite{Liu16}. Once this is clear, the rest follows from an application of Kawamata's base point free theorem. The general case can be reduced to the case of cone singularity using a degeneration process, which heavily relies on MMP techniques.
\bigskip

Now we consider the converse direction. 
For any klt singularity, a minimizer of the normalized volume function always exists by \cite{Blu16b}. The following theorem says that if a minimizer is divisorial, it always yields a Koll\'ar component. We can indeed prove slightly more for a general rational rank one minimizer. 
\begin{thmx}\label{t-divisor}
Given an arbitrary algebraic klt singularity $o\in (X,D)$ where $X=\Spec(R)$. Let $v$ be a valuation that minimizes $\hvol_{(X,D),o}$. Assume the valuation group of $v$ is isomorphic to $\mathbb{Z}$, i.e., $v$ has rational rank one,  and one of the following two assumptions holds
\begin{enumerate}
\item $v$ is a multiple of a divisorial valuation; or
\item the graded family of valuative ideals  $$\fa_{\bullet}=\{\fa_k\} \mbox{\ where \ } \fa_k=\{f\in R\ |\ v(f)\ge k\} $$ is finitely generated, i.e., there exists $m\in \mathbb{N}$ such that $\fa_{mk}=(\fa_m)^k$ for any $k\in \mathbb{N}$.
\end{enumerate}
Then up to a rescaling, $v$ is given by the divisorial valuation induced by a Koll\'ar component $S$.
\end{thmx}

The above theorem 
is also independently proved in \cite{Blu16b} by a different argument. 
We note that a minimizer is conjectured to be quasi-monomial and  the associated graded ring for a minimizer of the normalized volume function is conjectured to be always finitely generated (cf. \cite[Conjecture 7.1]{Li15a}). So granted these conjectures, the above result should presumably characterize all the cases with minimizers of rational rank 1. 
After giving the definition of the volume of a model, the proof uses similar MMP arguments to give a process decreasing the volumes as in \cite{LX14}.

\bigskip

Next we turn to the stability of the minimizer. By using the techniques from the toric degeneration (see \cite{Cal02, AB04, And13}) and the relation between the CM weight and normalized volume, we 
will prove
\begin{thmx}\label{t-mintok}
We use the same notation as in Theorem \ref{t-divisor}. Let $\mu\colon Y\to X$ be the morphism which extracts $S$, and write $(K_Y+S+\mu_*^{-1}D)|_S=K_S+\Delta_S$, then $(S,\Delta_S)$ is a K-semistable log Fano pair. 
\end{thmx}

\subsection{Approximation}

In a slightly different direction,  we also obtain results which describe the minimizer of the normalized  volume function from the viewpoint of Koll\'ar components. We show that 
for a general klt singularity, although the minimizer of its associated normalized volume function might not be given by one Koll\'ar component, we can always approximate it by a sequence. 
\begin{thm}\label{t-approx}
Given an arbitrary algebraic klt singularity $o\in (X,D)$, and a minimizer $v^{\rm m}$ of $\hvol_{(X,D),o}$, there always exists a sequence of Koll\'ar components $\{ S_j\}$ and positive numbers $c_j$ such that 
$$\lim_{j\to \infty}c_j\cdot \ord_{S_j}\to v^{\rm m}\mbox{\ in }\Val_{X,o} \mbox{\ \ \ and \ \ \ } \lim_{j\to \infty} \hvol(\ord_{S_j})=\hvol(v^{\rm m}).$$
\end{thm}
Here $\Val_{X,o}$ consists of all valuations centered at $o$, and is endowed with the weakest topology as in \cite[Section 4.1]{JM12}. See Remark \ref{r-link} for some discussions. 


\subsection{Equivariant K-semistability}
By relating a Fano variety to the cone over it, we can compare the calculation in \cite{Li15b} for a cone and \cite{Fuj16} for its base.  Then an interesting by-product of our method is the following theorem.
\begin{thmx}\label{t-equiK}
Let $T\cong (\mathbb{C}^*)^r$ be a torus. Let $(V,\Delta) $ be a log Fano variety with a $T$-action. Then $(V,\Delta)$ is K-semistable if and only if any $T$-equivariant special test configuration $(\mathcal{V}, \Delta^{\rm tc})\to \mathbb{A}^1$ of $(V,\Delta)$ has nonnegative generalized Futaki invariant: ${\rm Fut}(\mathcal{V}, \Delta^{\rm tc})\ge 0$. 
 \end{thmx}
 
 When $V$ is smooth and $\Delta=0$, this follows from the work of \cite{DS16} with an analytic argument. Our proof is completely algebraic. It again uses the techniques of degenerating any ideal to an equivariant one and showing that it has a smaller invariant. 

\bigskip

The paper is organized in the following way: In Section \ref{s-pre}, we give some necessary backgrounds.  In Section \ref{s-vmodel}, we introduce one key new tool: the volume of a model. By combining the normalized volume function on valuations with the local volume defined in \cite{Ful13}, and applying the MMP, we prove Theorem \ref{t-approx} and Theorem \ref{t-divisor}. In Section \ref{s-min}, we prove Theorem \ref{t-main}, by connecting it to the infimum of the normalized multiplicities $\lct(X,D;\fa)^n \cdot \mult(\fa)$ for all $\fm$-primary ideals $\fa$ centered on $o$. We note that this latter invariant indeed has also been studied in other context (cf. \cite{dFEM04}).  In Section \ref{s-uni}, we prove Theorem \ref{t-main2}. We first prove it for the cone singularity case, with the help of calculations from \cite{LL16}. Then we use a degeneration argument 
to reduce the general case to the case of cone singularities. In Section \ref{s-Ksta}, we prove Theorem \ref{t-mintok}, which verifies the K-semistability of a minimizing Koll\'ar component.  In Section \ref{s-exam}, we give some examples on how to apply our techniques to calculate the minimizer for various classes of klt singularities. 
 \vspace{5mm}
 
 \noindent {\it History:} Since \cite{Li15a}, there have been several papers related to the study of the minimization of the normalized volume function (see \cite{Li15b, LL16, Liu16, Blu16b, LX17, LWX18, BL18}).  In particular, after we posted the first version of our preprint, the existence of the minimizer is completely settled in \cite{Blu16b}. In the revision, we include his result in the exposition. We also get a complete characterization of K-semistability of $\bQ$-Fano varieties using the normalized volume, improving previous results from \cite{Li15b, LL16}.  Another major improvement in this revision is that we can indeed show Theorem \ref{t-mintok} that any Koll\'ar component which minimizes the normalized local volume is always K-semistable. 
  \vspace{5mm}
 
\noindent
{\bf Acknowledgement}: We thank Yuchen Liu, Dhruv Ranganathan and Xiaowei Wang for helpful discussions and many useful suggestions. We especially  want to thank Harold Blum and Mircea Musta\c{t}\v{a} for pointing out a gap in an earlier draft.  We also want to thank the referees for many valuable suggestions to improve the exposition. CL is partially supported by NSF DMS-1405936 and 
an Alfred P. Sloan research fellowship. CX is partially supported by `The National Science Fund for Distinguished Young Scholars (11425101)'. Part of the work was done when CX visited Imperial College London and Massachusetts Institute of Technology. He wants to thank Paolo Cascini and Davesh Maulik for the invitation and providing a wonderful environment.  The authors are grateful to referees for careful reading and suggestions on improving the paper.

\section{Preliminary}\label{s-pre}

\noindent{\bf Notation and Conventiones:} We follow the standard notation in  \cite{Laz, KM98, Kol13}. A {\it log Fano} pair $(V,\Delta)$ is a projective klt pair such that $-K_V-\Delta$ is ample.  

For a local ring $(R,\fm)$ and $\fa$ an $\fm$-primary ideal, we denote by $l_R(R/\fa)$ the length of $R/\fa$. 

For a variety $\bullet$, we sometimes denote the product $\bullet \times \mathbb{A}^1$ by $\bullet_{\mathbb{A}^1}$.

We will use interchangeably the notations $\bA^1$ with $\bC$, and $\bG_m$ with $\bC^*$.

\subsection{K-semistability}\label{ss-tc}
In this section, we give the definition of  K-semistability of a log Fano pair following \cite{Tia97, Don02} (also see \cite{Oda13,LX14}).

First we need to define the notion of {\it test configuration}.
\begin{defn}
 Let $(V,\Delta)$ be an $(n-1)$-dimensional log Fano pair. A ($\bQ$-)test configuration of $(V,\Delta)$ consists of
\begin{enumerate}
\item[ $\cdot$]
a pair $(\mathcal{V}, \Delta^{\rm tc})$ with a $\mathbb{G}_m$-action,
\item[ $\cdot$]
a $\mathbb{G}_m$-equivariant ample  $\mathbb{Q}$-line bundle $\mathcal{L}\rightarrow \mathcal{V}$,
\item[ $\cdot$]
a flat $\mathbb{G}_m$-equivariant map $\pi: \mathcal{V}\rightarrow \mathbb{A}^1$, where $\mathbb{G}_m$
acts on $\mathbb{A}^1$ by multiplication in the standard way
$(t,a)\to ta$ 
\end{enumerate}
such that for any $t\neq 0$, the restriction of  $(\mathcal{V}, \Delta^{\rm tc}, \mathcal{L})$ over $t$ is isomorphic to $(V, \Delta, -(K_V+\Delta))$, and $\Delta^{\rm tc}$ does not have any vertical component, i.e., components of $\Delta^{\rm tc}$ are the closures of components of $\Delta$ under the $\mathbb{G}_m$-action.

A test configuration $(\cV, \Delta^{\rm tc}, \cL)$ is called  special  if the central fibre $(V_0, \Delta_0)$ is a log Fano variety with klt singularities and $\cL\sim_{\bQ} -(K_{\cV}+\Delta^{\rm tc})$.
\end{defn}
By \cite{LX14}, without the loss of generality we will always assume that the test configuration is normal.  
Let $(\cV, \Delta^{\rm tc},\cL)$ be a test configuration of $(V, \Delta)$. Let $(\bar{\cV}, \bar{\Delta}, \bar{\cL})\rightarrow \bP^1$ be the natural compactification of $(\cV, \Delta^{\rm tc}, \cL)\rightarrow \bA^1$ by adding a trivial fibre $(V, \Delta, L)$ over $\{\infty\}\in \bP^1$. We call it a {\it compactified} test configuration. 
Then we can define the {\it  generalized Futaki invariant}.:
\begin{defn}
With the above notations, for any normal test configuration $(\cV, \Delta^{\rm tc}, \cL)$, we define its {\it generalized Futaki invariant} to be:
\begin{equation}
{\rm Fut}(\cV, \Delta^{\rm tc}, \mathcal{L})\left(={\rm Fut}(\bar{\cV}, \bar{\Delta}, \bar{\cL})\right)=\frac{1}{n (-K_V-\Delta)^{n-1}}((n-1) \bar{\cL}^n+n \bar{\cL}^{n-1}\cdot K_{\bar{\cV}/\bP^1}).
\end{equation}
In particular, for any special test configuration, we have:
\begin{equation}
{\rm Fut}(\cV, \Delta^{\rm tc}, \mathcal{L})=\frac{(-K_{\bar{\cV}/\bP^1}-\bar{\Delta})^{n}}{n (-K_V-\Delta)^{n-1}}.
\end{equation}
\end{defn}
The above definition of the generalized Futaki invariant using the intersection formula  is well known to be equivalent to the original one using the Riemann-Roch formula (cf. \cite{Wang12, Oda13, LX14} etc.).

\begin{defn}
A log Fano pair $(V,\Delta)$ is K-semistable if for any test configuration $(\mathcal{V}, \Delta^{\rm tc}, \mathcal{L})$ of $(V,\Delta)$, we have
$${\rm Fut}(\mathcal{V}, \Delta^{\rm tc}, \mathcal{L})\ge 0.$$
\end{defn}

\subsection{Normalized volume}

Let $(X,o)$ be a normal algebraic singularity and $D\ge 0$ be a $\mathbb{Q}$-divisor such that $K_X+D$ is $\mathbb{Q}$-Cartier. Denote by $\Val_{X,o}$ the space of real valuations centered on $o$. For any $v\in \Val_{X,o}$, we can define the volume $\vol_{X,o}(v)$ following \cite{ELS03} and the log discrepancy $A_{(X,D)}(v)$ following \cite{JM12, BFFU15} (if the context is clear, we will abbreviate it as $\vol(v)$ and $A(v)$).  In particular, if $S$ is a divisor over $X$, we have 
$$A_{(X,D)}(S):=A_{(X,D)}(\ord_S)=a(S; X,D)+1$$ 
which is the same as the standard log discrepancy. 

\begin{defn}\label{def-hvol}
Notation as above. We define the {\bf normalized volume}, denoted by $\hvol_{(X,D),o}(v)$ (or by $\hvol_{(X,D)}(v)$ if $o$ is clear or simply by $\hvol(v)$ if there is no confusion), to be 
$$\vol_{X,o}(v)\cdot A_{(X,D)}^n(v)$$ 
if $A_{(X,D)}(v)<+\infty$; and  $+\infty$  if $A_{(X,D)}(v)=+\infty$. 
We define the {volume} of a klt singularity $o\in (X,D)$ to be
$$\vol(o, X,D)=\inf_{v\in \Val_{X,o} } \hvol_{X,D}(v). $$ 
\end{defn}

\begin{rem}\label{r-link}
\begin{enumerate}
\item
The space $\Val_{X,o}$ is called the `non-archimedean link' of $o\in X$ in some literature (see \cite{Thu07, Fan14}). It was well known that in the topological setting the Euclidean link captures a lot of (including all the topological) information of a singularity. We expect that the study of $\Val_{X,o}$ will also significantly improve our knowledge of the singularity.

One can try to investigate the normalized volume function more globally. For instance, it is interesting to ask on a fixed model, how the function $\vol(o, X,D)$ changes when we vary $o$, including the case that $o$ is not a closed point. In particular, we expect that there is a formula to connect the volume of a (not necessarily closed) point $o$ and the volume of a general point $o'$ on the closure $\overline{\{o\}}$. We note that this may give us a way to treat those valuations with centers containing the fixed point. It is also natural to ask how $\vol(o, X,D)$ changes when we modify the birational models. We hope to explore these interesting questions in the future. 
\item
The volume of klt singularities defined here is different with the volume of singularities defined in \cite{BdFF12} (see also \cite{Zha14}). The volume in \cite{BdFF12} is defined using envelops of log discrepancy $b$-divisors and vanishes for klt singularities. Intuitively, while \cite{BdFF12} computes the volume of log canonical classes, our definition of volume of klt singularities is for the anti-log-canonical classes.  

\end{enumerate}
\end{rem}
 
 In \cite{Li15a}, it was shown that the space
$$\{v\in \Val_{X,o}|\ v(\fm)=1, \hvol(v) \le C \}$$
for any constant $C>0$ forms a compact set in a weak topology. However, in general the volume function $\vol$ is only upper semicontinuous on $\Val_{X,o}$. 

\begin{prop}\label{p-upperconti}
If $\{v_i\}$ is a sequence of valuations, such that $v_i\to v$ in the weak topology, then 
$$\vol(v)\ge \limsup_i \vol(v_i). $$ 
\end{prop}
\begin{proof}The valuation $v$ determines a graded sequence of ideas 
$$\fa_k=\fa_k(v)=\{f\in R\ |\  v(f)\ge k\}.$$  
By \cite{Mus02}, we know that for any $\epsilon>0$, there exists a sufficiently large $k$ such that 
$$\frac{1}{k^n}\mult(\fa_k)< \vol(v)+\epsilon.$$
Since $R$ is Noetherian, we know that there exist  finitely many generators $f_p$ ($1\le p\le j$) of $\fa_k=(f_1,...,f_j)$. As $v(f_p)\ge k$, we know that for any $\delta$, there exists sufficiently large $i_0$ such that for any $i\ge i_0$, $v_i(f_p)\ge k-\delta$. Thus 
$$\fa^{(i)}_{k-\delta}=\{f\in R\ |\  v_i(f)\ge k-\epsilon \}\supset \fa_k .$$
Therefore,
$$\vol(v_i)\le \frac{1}{(k-\delta)^n}\mult(\fa^{(i)}_{k-\epsilon})\le  \frac{1}{(k-\delta)^n}\mult(\fa_{k})\le  \frac{k^n}{(k-\delta)^n}(\vol(v)+\epsilon).$$
\end{proof}

We also have the following  result.
\begin{prop}\label{p-valuation}
Let $(X,o)=(\Spec(R),\fm)$ be a singularity. Let $v$ and $v'$ be two real valuations in $\Val_{X,o}$. Assume 
$$\vol(v)=\vol(v')>0\qquad \mbox{and}  \qquad v(h)\ge v'(h)$$ for any $h\in R$, then $v=v'$.
\end{prop}
\begin{proof} We prove it by contradiction. Assume that this is not true. We fix $f\in R$ such that 
$$v(f)=l>v'(f)=s.$$
Denote by $r= l-s>0.$ Fix $k\in \mathbb{R}_{>0}$. Consider 
$$\fa_k:=\{h\in R | \ v(h)\ge k\}\qquad\mbox{ and} \qquad\fb_k:=\{h\in R |\  v'(h)\ge k\}. $$
So by our assumption $\fb_k\subset \fa_k$, and we want to estimate the dimension of 
$$\dim(R/{\fb_k})-\dim(R/{\fa_k})=\dim (\fa_k/\fb_k). $$
Fix a positive integer $m<\frac{k}{l}$ and a set
$$g^{(1)}_{m},...,g^{(k_m)}_{m} \in \fb_{k-ml}$$
whose images in $\fb_{k-ml}/\fb_{k-ml+r}$ form a $\mathbb{C}$-linear basis.

We claim that 
$$\{f^m\cdot  g^{(j)}_{m}\}\ \ (1\le m \le \frac{k}{l}, 1\le j\le k_m)$$ 
are $\mathbb{C}$-linear independent in $\fa_k/\fb_k$. Granted this for now, 
we know that since $\vol(v)>0$ then 
$$\limsup_{k\to \infty} \frac{1}{k^n}\sum_{1\le m \le \frac{k}{l}}k_m=\limsup_{ k\to \infty} \sum_{1\le m\le \frac{k}{l}} \frac{1}{k^n} \dim (\fb_{k-ml}/\fb_{k-ml+r})>0, $$
which then implies $\vol(v)>\vol(v')$.

\bigskip

Now we prove the claim.

\noindent{\bf Step 1:} For any $1\le m \le \frac{k}{l}, 1\le j\le k_m$,
\begin{eqnarray*}
v(f^m\cdot g^{(j)}_{m})&=&v(f^m)+v(g^{(j)}_{m})\\
 &\ge& ml+v'(g^{(j)}_{m})\\
 &\ge & ml+k-ml\\
 &\ge &k.
 \end{eqnarray*}
 Thus $f^m\cdot g^{(j)}_{m}\in \fa_k$. 
 
 \vspace{3mm}
 
 \noindent{\bf Step 2:} If  $$\{f^m\cdot  g^{(j)}_{m}\}\ \ (1\le m \le \frac{k}{l}, 1\le j\le k_m)$$ are not $\bC$-linear independent in $\fa_k/\fb_k$,
 then there is an equality
 $$\sum_{m}h_m= b \in  \fb_k,$$
 where there exists $c_j\in \mathbb{C}$, such that 
 $$h_m=f^m\cdot \sum_{1\le j \le k_m} c_jg^{(j)}_{m}$$ 
  and some $h_m\neq 0$. Consider the maximal $m$, such that $h_m\neq 0$.
 Since
\begin{eqnarray*}
  v'(h_m)&=&v'(f^m\cdot \sum_{1\le j \le k_m} c_jg^{(j)}_{m}) \\
  &=& v'(f^m)+v'(\sum_{1\le j \le k_m} c_j g^{(j)}_{m})\\
&< & ms+k-ml+r\\
 &= & k-(m-1)l+(m-1)s,
 \end{eqnarray*}
 where the third  inequality follows from that 
 $$ \sum_{1\le j \le k_m} c_jg^{(j)}_{m} \notin \fb_{k-ml+r} .$$
 However, we have 
\begin{eqnarray*}
  v'(h_m)&=&v'(b-\sum_{j<m}h_j) \\
  &\ge & \min \big\{ v'(b), v'(h_1),....v'(h_{m-1}) \big\}\\
&= & \min_{1\le j \le m-1}\{ k , js+k-jl\} \\
 &=& k-(m-1)l+(m-1)s,
 \end{eqnarray*}
which is a contradiction. 
\end{proof}

Several results in our work depend on a relation between normalized volumes of valuations and normalized multiplicities of primary ideals. The latter quantity was first considered in the smooth case in \cite{dFEM04}, and since then it has been studied  in many other works, including its positive characteristic version (see e.g. \cite{TW04}). Its relevance to the normalized volume appeared in \cite[Example 5.1]{Li15a}. In \cite{Liu16} the following more precise observation is made. 
\begin{prop}[{\cite[Section 4.1]{Liu16}}]\label{p-inf}
Let $(X,o)=({\rm Spec}R, \fm)$ and $D\ge 0$ a $\mathbb{Q}$-divisor, such that $o\in (X,D)$ is a klt singularity. Then we have 
\begin{equation}\label{eq-vol2mul}
\inf_{v} \hvol_{(X,D),o}(v)=\inf_{\fa}\lct^n(X,D; \fa)\cdot\mult(\fa),
\end{equation}
where on the left hand side $v$ runs over all real valuations centered at $o$, and on the right hand side $\fa$ runs over all $\fm$-primary ideals. Moreover, 
the left hand side can be replaced by $\inf_{v\in {\rm Div}_{X,o}} \hvol_{(X,D),o}(v)$ where ${\rm Div}_{X,o}$ denotes the space of all divisorial valuations with center at $o$.
\end{prop}
For the reader's convenience we provide a sketch of the proof. 

\begin{proof}
We first use the same argument as in \cite[Example 5.1]{Li15a}) to prove that the left hand side is greater than or equal to the right hand side. For any real valuation $v$, consider 
the graded family of valuative ideals 
$$\fa_k=\fa_k(v)=\{f\in R\ |\  v(f)\ge k\}.$$ 
Then $v(\fa_k)\ge k$ and we can estimate:
\begin{eqnarray*}
A_{(X,D)}(v)^n \cdot \frac{\mult(\fa_k)}{k^n}&\ge & \left(\frac{A_{(X,D)}(v)}{v(\fa_k)}\right)^n \cdot \mult(\fa_k)\ge \lct^n(X,D; \fa_k)\cdot \mult(\fa_k).
\end{eqnarray*}
Since $\fa_\bullet=\{\fa_k\}$ is a graded family of $\fm$-primary ideals on $X$, 
\[
\vol(v)=\mult(\fa_\bullet)=\lim_{k\rightarrow+\infty} \frac{l_R(R/\fa_k)}{k^n}=\lim_{k\rightarrow+\infty}\frac{\mult(\fa_k)}{k^n}.
\]
(see e.g. \cite{ELS03, Mus02, LM09, Cut12}). As $k\rightarrow +\infty$, the left hand side converges to $\hvol(v)$ and we get one direction.

For the other direction of the inequality, we follow the argument in \cite{Liu16}. For any $\fm$-primary ideal $\fa$, we can choose a divisorial valuation $v$ calculating 
$\lct(\fa)$. Then $v$ is centered at $o$. Assume $v(\fa)=k$, or equivalently $\fa\subseteq \fa_k(v)$. Then we have $\fa^l \subseteq \fa_k(v)^l\subseteq \fa_{kl}(v)$ for any $l\in \bZ_{>0}$. So we can estimate:
\begin{eqnarray*}
\lct^n(X,D; \fa)\cdot \mult(\fa)&=&\frac{A_{(X,D)}(v)^n}{k^n}\cdot \mult(\fa)=A_{(X,D)}(v)^n\cdot\frac{\mult(\fa) l^n}{(kl)^n}\\
&=&A_{(X,D)}(v)^n \cdot \frac{\mult(\fa^l)}{(kl)^n}\ge A_{(X,D)}(v)^n\cdot \frac{\mult(\fa_{kl})}{(kl)^n}.
\end{eqnarray*}
As $l\rightarrow +\infty$, then again the right hand side converges to 
$$A_{(X,D)}(v)^n \cdot \mult(\fa_\bullet(v))=\hvol(v).$$ 

The last statement follows easily from the above proof.
 \end{proof}
 
In \cite{Blu16b}, it is proved that a minimizer always exists.

\begin{thm}[\cite{Blu16b}] For any klt singularity $o\in (X,D)$, $\hvol_{(X,D)}(v)$ always has a minimizer $v^{\rm m}$ in $\Val_{X,o}$.
\end{thm}

\subsection{Properties of Koll\'ar component}
The concept of {\it Koll\'ar component} is defined in Definition \ref{d-kollar}. It always exists by results from the MMP (cf. see \cite{Pro00} or \cite[Lemma 1]{Xu14}).  

In this section, we  establish some of their properties using the machinery of the minimal model program. The following statement is the local analogue of \cite[Theorem 1.6]{LX14}, which can be obtained by following the proof of the existence of Koll\'ar component. (See e.g. the proof of \cite{Xu14}.) 
\begin{prop}\label{p-special}Let $o\in (X,D)$ be a klt singularlty. Let $\mu\colon Z\to X$ be a model, such that $\mu$ is an isomorphism over $X\setminus\{o\}$ and $(Z,E+\mu_*^{-1}D)$ is dlt where $E$ is the divisorial part of $\mu^{-1}(o)$. 
Then we can choose a model $W\to Z$ and  run an MMP to obtain $W\dasharrow Y$, such that $Y\to X$ gives a Koll\'ar component $S$ that satisfies $a(S; Z,E)=-1$. 
\end{prop}

We also have the following straightforward lemma. 
\begin{lem}\label{l-inter}
If $S$ is a Koll\'{a}r component as the exceptional divisor of a plt blow-up $\mu: Y\rightarrow X$, then $\vol(\ord_S)=(-S|_S)^{n-1}$ and $\hvol(\ord_S)=(-(K_Y+S+\mu_*^{-1}D)|_S)^{n-1}\cdot A_{(X,D)}(S)$.
\end{lem}
\begin{proof}
For any $k\ge 0$ such that $kS$ is Cartier on $Y$, we have an exact sequence,
$$0\to \cO_{Y}(-(k+1)S)\to \cO_{Y}(-kS)\to \cO_{S}(-kS)\to 0.$$
Because $-S$ is ample over $X$, we have the vanishing 
$$R^1f_*(\cO_{Y}(-(k+1)S))=0,$$
from which we get 
\[
H^0(S, -kS|_S)\cong \frac{H^0(Y, -kS)}{H^0(Y, -(k+1)S)}=\frac{\fa_k(\ord_S)}{\fa_{k+1}(\ord_S)}.
\]
for any such $k$.
Then the result follows easily from the Hirzebruch-Riemann-Roch formula and the asymptotic definition of $\vol(\ord_S)$.

As $K_Y+S+\mu^{-1}_*D\sim_{\mathbb{Q},X} A_{(X,D)}(S)\cdot S$, the second identity is implied by the first statement. 
\end{proof}

\begin{rem}Inspired by the above simple calculation, we can indeed extend the definition of normalized volumes to any model $f:Y\to (X,o)$, such that $f$ is isomorphic over $X\setminus \{o\}$. See Section \ref{s-vmodel}. 
\end{rem}

\begin{lem}\label{l-finite}
Let $f\colon (X',o')\to (X,o)$ be a finite morphism, such that $f^*(K_X+D)=K_{X'}+D'$ for some effective $\mathbb{Q}$-divisors. We assume  $(X,D)$ and $(X',D')$ are klt. If $S$ is a Koll\'ar component given by $Y\to X$ over $o$, then $Y':=Y\times_XX'\to X'$ induces a Koll\'ar component $S'$ over $o'\in (X',D')$. 

Conversely, if $X'\to X$ is Galois with Galois group $G$, then any $G$-invariant Koll\'ar component $S'$ over $o\in (X',D')$ is the pull back from a Koll\'ar component over $o\in (X,D)$.
\end{lem}
\begin{proof} The first part is standard. Denote by $\mu'\colon Y'\to X'$ the birational morphism and by $S'=(f_Y^{-1}(S))_{\rm red}$ where $f_Y\colon Y'\to Y$ is the induced morphism, then $(Y',\mu'^{-1}_*D'+S')$ is log canonical. If we restrict to $T$ a component of $S'$, 
$$(K_{Y'}+\mu'^{-1}_*D'+S')|_{T}=K_{T}+\Delta_{T},$$
then $(T,\Delta_{T})$ is klt, which by Koll\'ar-Shokurov connectedness theorem implies that $T=S'$.

For the converse, let  $$L\sim_{X'}-m(K_{Y'}+\mu'^{-1}_*D'+S')$$  be a divisor of general position for sufficiently divisible $m$ and  $H:=\frac{1}{m}L$, then $(Y', S'+\mu'^{-1}_*D'+H)$ is plt. Replacing $H$ by $H_G:=\frac{1}{|G|}(\sum_{g\in G} g^*H)$, we know that $(X', D'+ \mu_*H_{G})$ is $G$-invariant, and there exists a $\mathbb{Q}$-divisor $H_X\ge 0$, such that
$$f^*(K_X+D +H_X)=K_{X'}+D'+\mu_*H_G.$$
Therefore, $(X,D+H_X)$ is plt, and its unique log canonical place is a divisor $S$ which is a Koll\'ar component over $o\in (X,D)$ whose pull back gives the Koll\'ar component $S'$ over $o'\in (X',D')$.  
\end{proof}

We prove a change of volume formula for Koll\'ar components under a finite map.
\begin{lem}\label{l-finitevolume}With the same notation as in Lemma \ref{l-finite}, then 
$$d\cdot \hvol_{(X,D)}(\ord_S)=\hvol_{(X',D')}(\ord_{S'}),$$
where $d$ is the degree of $X'\to X$. 
\end{lem}
\begin{proof}
The pull back of $S$ is $S'$ which is irreducible by Lemma \ref{l-finite}. Let the degree of $S'\to S$ be $a$ and the ramified degree be $r$. We have the identity: 
$$ar=d \qquad \mbox{and}\qquad rA_{(X,D)}(\ord_S)=A_{X',D'}(\ord_{S'})$$ 
(see \cite[5.20]{KM98}).
By Lemma \ref{l-inter}, we know that 
\begin{eqnarray*}
d\cdot \hvol_{(X,D)}(\ord_S)& = & ar\cdot A_{(X,D)}(\ord_S)\cdot ((K_Y+S+\mu_*^{-1}D)|_{S})^{n-1}\\
&=&(rA_{(X,D)}(\ord_S))\cdot \big( a\cdot((K_Y+S+\mu_*^{-1}D)|_{S})^{n-1}\big)\\
&= &A_{(X',D')}(\ord_{S'})\cdot \big( ((K_{Y'}+S'+\mu_*'^{-1}D')|_{S'})^{n-1}\big)\\
&=&\hvol_{(X',D')}(\ord_{S'}),
\end{eqnarray*}
where for the third equality we use the projection formula of intersection numbers. 
\end{proof}



\subsection{Deformation to normal cones}\label{ss-deformation}
Let $(X, o)=({\rm Spec}(R), \fm)$ be an algebraic singularity such that $(X,D)$ is klt for a $\mathbb{Q}$-divisor $D\ge 0$. Let $S$ be a Koll\'{a}r component and $\Delta=\Delta_S$ be the different divisor defined by the adjunction $(K_Y+S+\mu_*^{-1}D)|_S=K_S+\Delta_S$ where $Y\rightarrow X$ is the extraction of $S$ (see \eqref{eq-diffS}).

For simplicity denote $v_0:=\ord_S$. Also denote

\begin{equation}\label{eq-R*}
R^*:=\bigoplus_{k=0}^{+\infty} \fa_k(v_0)/\fa_{k+1}(v_0)=\bigoplus_{k=0}^{+\infty} R^*_k
\end{equation} 
and its $d$-th truncation
$$R^{*(d)}=\bigoplus_{k=0}^{+\infty} \fa_{dk}(v_0)/\fa_{dk+1}(v_0)=\bigoplus_{k=0}^{+\infty} R^*_{dk}\qquad \mbox{for $d\in \mathbb{N}$}.$$

Now we give a more geometric description of $\Spec(R^*)$ and $\Spec(R^{*(d)})$ using the ideal of degenerating  $o\in (X,D)$ to an (orbifold) cone over the Koll\'ar component $S$. Assume $\mu\colon Y\rightarrow X$ is the extraction of the Koll\'{a}r component $S$ of $(X, o)$. Then 
$\mu_{\mathbb{A}^1}\colon Y\times\bA^1\rightarrow X\times \bA^1$ has the exceptional divisor $S\times \bA^1$. The divisor $S$ is not necessarily Cartier, but only $\mathbb{Q}$-Cartier. Thus we can take the index one covering Deligne-Mumford stack $\pi: \mathfrak{Y}\to Y$ for $S$. So $\pi$ is isomorphic over $Y\setminus S$ and $\pi^*(S)=\mathfrak{S}$ is Cartier on $\mathfrak{Y}$. Note that $S$ and $Y$ are coarse moduli spaces of $\mathfrak{S}$ and $\mathfrak{Y}$ respectively.

We consider the deformation to the normal cone construction for $\mathfrak{S}\subset \mathfrak{Y}$ (see \cite[Chapter 5]{Ful84}). More precisely, we consider the blow up $ \tilde{\phi}_1\colon \mathfrak{Z} \rightarrow \mathfrak{Y}\times \bA^1$ along $\mathfrak{S}\times\{0\}$. Denote by $\mathfrak{P}$ the exceptional divisor and by $\mathfrak{S}'_{\bA^1}$ the strict
transform of $\mathfrak{S}\times\bA^1$. We note that $\mathfrak{P}$ has a stacky structure along the 0 and $\infty$ section, but a scheme structure at other places. 
Then  $\mathfrak{S}'_{\bA^1}\subset \mathfrak{Z}$ is a Cartier divisor which is proper over $\bA^1$ and can be contracted to a normal Deligne-Mumford stack $\tilde{\psi}_1\colon \mathfrak{Z}\to \mathfrak{W}$ and in this way we get a flat family $\mathfrak{W}\rightarrow \bA^1$ such that $\mathfrak{W}_t\cong X$ and $\mathfrak{W}_0\cong \bar{\mathfrak{C}}\cup \mathfrak{Y}_0$, where $\mathfrak{Y}_0$ is the birational transform of $Y\times\{0\}$. If we denote by $\mathfrak{W}^\circ:=\mathfrak{W}\setminus \mathfrak{Y}_0$, then the fiber $\mathfrak{W}^\circ$ over $0$ is isomorphic to $\mathfrak{C}$ which is an affine orbifold cone over $\mathfrak{S}$ with the polarization given by $\mathcal{O}_{\mathfrak{Y}}(-\mathfrak{S})|_{\mathfrak{S}}$.
Moreover, $\bar{\mathfrak{C}}$ is the projective orbifold cone completing $\mathfrak{C}$. We will also denote by $\cW$, $\cW^\circ$, $\cZ$, $\bP$ the underlying coarse moduli spaces of $\mathfrak{W}$, $\mathfrak{W}^\circ$, $\mathfrak{Z}$, $\mathfrak{P}$ respectively. In particular, we have (see figure \ref{fig-deg}):
\begin{equation}\label{eq-fibers}
\def\arraystretch{1.2}
\begin{array}{l}
\cZ\times_{\bA^1} (\bA^1\backslash \{0\})=Y\times(\bA^1\backslash\{0\}), \quad \cZ\times_{\bA^1}\{0\}=\bP\cup Y_0.\\
\cW\times_{\bA^1} (\bA^1\backslash\{0\})=X\times(\bA^1\backslash\{0\}), \quad \cW\times_{\bA^1}\{0\}=\bar{C}\cup Y_0. \\
\cW^\circ \times_{\bA^1} (\bA^1\backslash\{0\})=X\times(\bA^1\backslash\{0\}), \quad \cW^\circ\times_{\bA^1}\{0\}=C.
\end{array}
\end{equation}

Let $d$ be a positive integer such that $d\cdot S$ is Cartier in $Y$, then $\mathfrak{C}^{(d)}$ given by the cone over $\mathcal{O}_{\mathfrak{Y}}(-d\cdot\mathfrak{S})|_{\mathfrak{S}}$ is a degree $d$ cyclic quotient of $\mathfrak{C}$, which is a usual ($\mathbb{A}^1$-)cone over $\mathfrak{S}$.
We denote by $C$ and $C^{(d)}$  the underlying coarse moduli space of $\mathfrak{C}$ and $\mathfrak{C}^{(d)}$. We also denote by $S$ the coarse moduli space of $\mathfrak{S}$. The vertex of $C$ is denoted by $o_C$.

For any $k$ such that $kS$ is Cartier, applying the exact sequence, 
$$0\to \mathcal{O}_Y(-(k+1)S)\to \mathcal{O}_Y(-kS)\to \mathcal{O}_{S}(-kS)\to 0,$$
since $h^1(\mathcal{O}_Y(-(k+1) S))=0$ by the Grauert-Riemenschneider vanishing theorem, we get:
\[
H^0(S, \mathcal{O}(-k S|_S))\cong H^0(\cO_Y(-kS))/H^0(\cO_Y(-(k+1)S)).
\]
Notice that the right hand side is equal to:
\[
\frac{\mu_*\cO_Y(-kS)}{\mu_*\cO_Y(-(k+1)S)}=\frac{\fa_{k}(v_0)}{\fa_{k+1}(v_0)}.
\]
In particular, $C^{(d)}={\rm Spec}(R^{*(d)})$. Similarly, we have $C={\rm Spec} (R^*)$.

There is also a degree $d$ cyclic quotient morphism $h\colon \bar{C}\to \bar{C}^{(d)}$, and we know that
$$h^*(K_{\bar{C}^{(d)}}+C^{(d)}_{1}+C^{(d)}_{2})=K_{\bar{C}}+C_{D},$$ 
where $C_D$ is the intersection of $\bar{C}$ with the birational transform of $D\times \mathbb{A}^1$ and $C^{(d)}_{1}$ (resp. $C^{(d)}_{2}$) on $\bar{C}^{(d)}$ is the induced cone over the branched $\mathbb{Q}$-divisor  on $S$ of $\mathfrak{S}\to S$ (resp. $\mu^{-1}_*D|_S$). 

\subsection{Filtrations and valuations}\label{sec-filtration}

Here we recall some facts about $\bZ$-graded filtration and its relation to valuations following the work in \cite{TW89}. 
A filtration on $R$ is a decreasing sequence $\cF:=\{\cF^m\}_{m\in \bZ}$ of ideals of $R$ satisfying the following conditions:

{\bf (i)} $\cF^m\neq 0$ for every $m\in \bZ$, $\cF^m=R$ for $m\le 0$ and $\cap_{m\ge 0}\cF^m=(0)$.

{\bf (ii)} $\cF^{m_1}\cdot \cF^{m_2}\subseteq \cF^{m_1+m_2}$ for every $m_1, m_2\in \bZ$.

\vskip 2mm
Notice we can replace the grading $\bZ$ by any abelian group that is isomorphic to $\bZ$. For a given filtration, we have the Rees algebra and extended Rees algebra:
\begin{equation}
\cR:=\cR(\cF)=\bigoplus_{k=0}^{+\infty} (\cF^k R) t^{-k}, \quad \cR':=\cR'(\cF)=\bigoplus_{k=-\infty}^{+\infty} (\cF^k R) t^{-k},
\end{equation}
and the associated graded ring:
\begin{equation}
\gr_\cF(R)=\cR'/t \cR'=\bigoplus_{k=0}^{+\infty} (\cF^k R/\cF^{k+1} R) t^{-k}.
\end{equation}
Assuming $\cR'$ is finitely generated, $\cX:={\rm Spec}_{\bC[t]} (\cR')$ can be seen as a $\bC^*$-equivariant flat degeneration of $X={\rm Spec}(R)$ into 
$\cX_0={\rm Spec}_{\bC} (\cR'/t\cR')={\rm Spec}_{\bC}(\gr_\cF R)$. Denote $E={\rm Proj} (\gr_\cF(R))$, $\tilde{X}={\rm Proj}_{R} \cR$. Then the natural map $\tilde{X}\rightarrow X$ is the filtered blow up associated with the $\cF$ such that $E$ is the exceptional divisor. Moreover $\tilde{X}$ can be seen as a flat deformation of a natural filtered blow up on $\cX_0$. Indeed following \cite[5.15]{TW89}, we have a filtration $\cF$ on $\cR'$:
\[
\cF^m \cR'=\left\{\sum_{k=-\infty}^{+\infty} \left(\cF^{\max(k,m)} R\right) t^{-k}  \right\}.
\]

The objects associated to the corresponding Rees algebra and graded algebra over $\cR'$ are:
\[
\tilde{\cX}={\rm Proj}_{\cR'} \bigoplus_{r=0}^{+\infty}(\cF^r \cR' ) T^{-r}, \quad \mathcal{E}={\rm Proj}_{\bC} \bigoplus_{r=0}^{+\infty} (\cF^r \cR' /\cF^{r+1} \cR') T^{-r}.
\]
Moreover, since $\cR'$ is finitely generated, there is an embedding  $X\subset \mathbb{C}^p$ for some $p\in\mathbb{N}$ given by $x_i\to f_i$ $(i=1,...,p)$ where $f_1,...,f_p$ is a set of elements with $f_i\in \mathcal{F}^{k_i}R$  such that $t^{-k_i}f_i$ $(i=1,...,p)$ and $t$  generates $\cR'$. Set ${\rm deg}(x_i)=k_i$ and let $\widehat{\bC^p}\to \bC^p$ be the weighted blow-up with weights $(k_1, \dots, k_p)$. 

Then we have the following commutative diagram (see \cite[Proposition 5.17]{TW89}):
\[
\begin{CD}
&& E @>>> \cE @<<< E \\
&& @VVV @VVV @VVV \\
\widehat{\bC^p}@<<<  \tilde{X} @>>> \tilde{\cX} @<<< \tilde{\cX}_0\\
@VVV  @VVV @VVV @VVV\\
\bC^p @<<< X@>>> \cX  @<<< \cX_0.
\end{CD}
\]

To relate the filtrations to valuations, the need the following well-known fact:
\begin{lem}[{see \cite[Page 8]{Tei14}}]\label{lem-fil2val}
If the associated graded ring of $\cF$ is an integral domain, then the filtration $\cF$ is induced by a valuation. 
\end{lem}
\begin{proof}
We define the order function $v: R\rightarrow \bZ$ by $v(f)=\max\{m; f\in \cF^m\}$. Then by the defining properties of filtrations, $v$ satisfies $v(f+g)\ge \min\{v(f), v(g)\}$ and $v(fg)\ge v(f)+v(g)$ for any $f,g\in R$. For any $f\in R$, let $[f]$ denote the image of $f\in R$ under the quotient map $\cF^{v(f)} \rightarrow \cF^{v(f)}/\cF^{v(f)+1}\subset \gr_{\cF}R$. Then $[f]\cdot [g]\neq 0$ by the assumption that $\gr_{\cF}$ is an integral domain. This translates to $v(fg)=v(f)+v(g)$ which implies $v$ is indeed a valuation.

\end{proof}

Actually we can be more precise in a special case that we will deal with later. There is a natural $\bC^*$-action on $\cX_0$ associated to the natural $\mathbb{N}$-grading such that the quotient is isomorphic to $E$. Let $\mathcal{J}=\bigoplus_{k\ge 0} \cF^{k+1}t^{-k}=t \cR'\cap \cR$ so that $\cR/\mathcal{J}\cong \gr_{\cF}(R)\cong \cR'/t \cR'$.
Now we assume furthermore that 
$E$ is a normal projective variety. This implies both $\cR$ and $\cR'$ are normal (see \cite{TW89}). 
Let $\mathfrak{P}$ be the unique minimal prime ideal of $\cR$ over $\mathcal{J}$ that corresponds to the cone over ${E}$, and $w$ the valuation of $K(t)$ attached to $\mathfrak{P}$.  Then the restriction of $w$ to $R$ is equal to $b \cdot \ord_E$. Assume $a=w(t)$. Thus the filtration $\cF$ is equivalent to the filtration that is given by:
\[
(t^m \cR') \cap R=\{f\in R; \ord_E(f)\ge m  a/b\}.
\]
\begin{rem}
There is a general Valuation Theorem about  the relation between finitely generated filtrations and valuations proved by Rees for which we refer the reader to \cite{Ree88}. See also \cite{BHJ15}.
\end{rem}



\section{Volume of models}\label{s-vmodel}

One very useful tool for us to study the minimzer of the normalized local volume is the concept of  a local volume of a model. It is this concept which enables us to apply the machinery of the minimal model program to construct different models, especially those yielding Koll\'ar components.  

\subsection{Local volume of models}\label{s-lvmodel}
In this section, we extend the definition of volume to volumes of birational models in the `normalized' sense. We use the concept of local volumes as in \cite{ELS03, Ful13}. Let us first recall the definition, which is from \cite{Ful13}.
\begin{defn}[Local volume](cf. See \cite{Ful13})
Let $X$ be a normal algebraic variety of dimension $n\ge 2$ and let $o$ be a point on $X$. For a fixed  a proper birational map $\mu\colon Y \to  X$ and a Cartier divisor $E$ on $Y$,  we define the local volume of $E$ at $o$ to be
$$\vol^F_o(E) = \limsup_{m\to \infty} \frac{h^1_o(mE)} {m^n /n!},$$
where $h^1_o(mE):=\dim H^1_{\{o\}}(X, \mu_*\cO_Y(mE)).$

If $E$ is a $\mathbb{Q}$-Cartier divisor, we define its volume to be
$$\vol^F_{o}(E):=\frac{\vol^F_o(mE)}{m^n},$$
for sufficiently divisible $m$.

\end{defn}

\begin{lem}\label{l-pushforward}
Let $\mu\colon Y\to X$ be a birational morphism. If $E\ge 0$ is an exceptional $\mathbb{Q}$-divisor, such that ${\rm Supp}(E)\subset \mu^{-1}(o)$, then 
$$\vol^F_{o}(-E)=\limsup_{k\to \infty} \frac{l_R(\mathcal{O}_X/\fa_k)}{k^n/n!},$$
where $k$ is sufficiently divisible and $\fa_k=\mu_*(\cO_Y(-kE))$.
\end{lem}
\begin{proof}This follows from \cite[Remark 1.1(ii)]{Ful13} (see also \cite[Remark 1.31 and 1.32]{Ful13}).
\end{proof}
The right hand side of the above display is also the volume ${\rm vol}(\fa_{\bullet})$ defined in \cite[Definition 3.1, Proposition 3.11]{ELS03}. In particular, given a prime divisor $E$ over $o$ with log discrepancy $a$, we see that 
$$\vol^{F}_o(-aE)=\hvol_{(X,D),o}(\ord_E).$$

\begin{defn}\label{d-modelv}
Assume that $o\in (X,D)$ is a klt singularity, and $\mu\colon Y\to (X,o)$ is a birational morphism such that  $\mu$ is an isomorphism over $X\setminus \{o\}$. Let $E=\sum_i G_i$ be the reduced divisor supported on the divisorial part of $\mu^{-1}(o)$. Then we define the volume $\vol_{(X,D),o}(Y)$ (abbreviated as $\vol(Y)$ if $(X,D; o)$ is clear) of $Y$ to be 
$$\vol_{(X,D),o}(Y):=\vol^F_{o}(-K_Y-E-\mu_*^{-1}(D))=\vol^F_{o}\left(\sum_i -a_iG_i\right),$$
where $a_i=A_{(X,D)}(G_i)$ is the log discrepancy of $G_i$.
\end{defn}

We will mainly combine the above definition with the following construction.
\begin{defn}\label{d-dlt}
For a klt pair $(X,D)$ with an ideal $\fa$, if $c$ denotes its log canonical threshold $\lct(X,D;\fa)$, then we say that $\mu\colon Y\to X$ is a {\it dlt modification} of $(X,D+c\cdot \fa)$, if the following conditions are all satisfied:
\begin{enumerate}
\item denote the divisorial part of $\mu^*(\fa)$ by $\mathcal{O}(-\sum m_iG_i)$ and denote by $\mu^*(K_X+D)=K_Y+D_Y$, then 
$$D_Y+c\cdot \sum m_i G_i=\mu^{-1}_*(D)+E$$ where $E$ is the reduced  divisor on ${\rm Ex}(\mu)$; 
\item $(Y,D_Y+c\cdot \sum m_i G_i)$ is dlt.
\end{enumerate}
By the argument in \cite{OX12}, we know that it follows from the MMP results in \cite{BCHM10} that a dlt modification of $(X,D+c\cdot \fa)$ always exists. More concretely, we can choose general elements $f_j\in\fa$ $(1\le j \le l)$ which generate $\fa$ such that $ \frac{c}{l}<1$. If we let $D_j=(f_j=0)$, then $Y$ is the dlt modification of $(X,D+c\cdot\frac{1}{l}\sum^l_{j=1}D_j)$. 
\end{defn}
\begin{lem}\label{l-gooddlt}
We can indeed assume that  $-K_Y-{\mu}_*^{-1}D-E$ is nef over $X$.
\end{lem}
\begin{proof}Since $(X,D)$ is klt, we know that 
$$K_Y+\mu_*^{-1}D+E\sim_{\mathbb{Q},X} \sum a_iG_i$$
with $G_i$ are all exceptional and $a_i=A_{(X,D)}(G_i)>0$. Running a relative MMP of 
$$(Y,\mu^{-1}_*(D+c\cdot \frac{1}{l}\sum_{j=1}^l D_j)+E-\epsilon \sum a_i G_i) \mbox{ \ \ over $X$}$$ with scaling by an ample divisor, we obtain a relative minimal model $Y\dasharrow Y'$ of 
$$
K_Y+\mu^{-1}_*(D+c\cdot \frac{1}{l}\sum_{j=1}^l D_j)+E \sim_{\bQ, X} - c\cdot \sum m_i G_i+\sum_i a_i G_i=0.
$$
So we have
$$
K_Y+\mu^{-1}_*(D+c\cdot \frac{1}{l}\sum_{j=1}^l D_j)+E-\epsilon \sum a_i G_i=-\epsilon (K_Y+\mu_*^{-1}D+E),
$$
and hence $-K_{Y'}-{\mu'}_*^{-1}D-E'$ is nef over $X$ where $\mu'\colon Y'\to X$ and $E'$ is the birational transform of $E$. Furthermore, since  
$$K_Y+\mu_*^{-1}D+E\sim_{\mathbb{Q},X}-c \cdot\frac{1}{l} \mu_*^{-1} \sum^l_{j=1}D_j,$$
$Y'$ also gives a minimal model of the dlt pair $$\big(Y, \mu_*^{-1}(D+c(1+\epsilon) \cdot\frac{1}{l}\sum^l_{j=1}D_j)+E\big),$$
which implies $(Y', {\mu'}_*^{-1}D+E')$ is a dlt modification of $(X,D+c\cdot\frac{1}{l}\sum^l_{j=1}D_j)$. Therefore, we can replace $ Y$ by $Y'$.
\end{proof}

When $E$ is irreducible, then $\vol_{(X,D),o}(Y)=\hvol_{(X,D),o}(\ord_E)$.  We can generalize Lemma \ref{l-inter} to the dlt case.
\begin{lem}\label{l-inter2}In the setting of Definition \ref{d-modelv}, if we assume that $-K_Y-\mu_*^{-1}D-E$ is nef over $X$. 
Then 
 $$\vol_{(X,D),o}(Y)=\sum_i a_i \big((-K_Y-\mu_*^{-1}D-E)|_{E_i}\big)^{n-1}.$$
\end{lem}
\begin{proof}Let $m$ be sufficiently divisible such that $L:=m(K_Y+\mu_*^{-1}D+E)$ is Cartier. Denote by $F$ the effective Cartier divisor $F:=\sum_i ma_iG_i$. 
Then 
$$0\to \mathcal{O}_Y(-(k+1) L)\to \mathcal{O}_Y(-k L)\to \mathcal{O}_F(-k L)\to  0.$$
Since $-L$ is nef, we know that $R^1\mu_*(\cO_Y(-(k+1)L))=0$.
Thus 
$$\vol_o^F(L)=\vol(L|_{F}),$$
and then we conclude by dividing $m^n$ in both sides. 
\end{proof}

\begin{lem}\label{l-model1}
Let $\fa$ be an $\fm$-primary ideal. Denote $c=\lct(X,D;\fa)$ and let $(Y,E)\to X$ be a dlt modification of $ (X,D+c\cdot \fa)$. Then
$$\vol_{(X,D),o}(Y)\le \lct^n(\fa )\cdot\mult(\fa ).$$
\end{lem}
\begin{proof}Write $K_Y+\mu_*^{-1}D+E=\mu^*(K_X+D)+\sum_i a_iG_i,$ where $E$ is the reduced divisor on ${\rm Ex}(\mu)$.  If we denote the vanishing order of $\mu^* \fa$ along $G_i$ by $m_i$, then since  $c$ is the log canonical threshold and for every $i$, $G_i$ computes the log canonical threshold,  we know that $c\cdot m_i=a_i$. Thus
$$\fa^k\subset \mu_*\mathcal{O}_Y(-\sum_i km_iG_i)=_{\rm def} \fb_k.$$
It suffices to show that 
$$\mult(\fb_{\bullet})=\vol_o^F (-\sum m_iG_i).$$
But this follows from Lemma \ref{l-pushforward}.
\end{proof}
\begin{lem}\label{l-kollar}
With the same assumptions as in Lemma \ref{l-model1}, there exists a Koll\'ar component $S$, such that 
$$\hvol_{(X,D),o}(\ord_S)\le \vol_{(X,D),o} (Y)\le \lct^n(\fa)\cdot \mult(\fa).$$ 
\end{lem}
\begin{proof}It follows from Proposition \ref{p-special} that we can choose a model $W\to Y$ and  run MMP to obtain $W\dasharrow Y'$, such that $\mu'\colon Y'\to X$ gives a Koll\'ar component $S$ with $a(S; Y,E+\mu_*^{-1}D)=-1$. 
If we fix a common resolution $p\colon W'\to Y $ and $q\colon W'\to Y'$, then since  $-(K_Y+E+\mu_*^{-1}D)$ is nef and $A_{Y, E+\mu_*^{-1}D}(S)=0$, 
we know $-p^*(K_Y+E+\mu_*^{-1}D)+q^*(K_{Y'}+S+\mu'^{-1}_*D)$ is $q$-nef and $q$-exceptional. By the negativity lemma, we get
$$p^*(K_Y+E+\mu_*^{-1}D)\ge q^*(K_{Y'}+S+\mu'^{-1}_*D).$$
Thus 
$$\hvol(\ord_S)=\vol(-K_{Y'}-S-\mu'^{-1}_*D)\le \vol(-K_Y-E-\mu_*^{-1}D)=\vol(Y).$$
\end{proof}

\subsection{Approximating by Koll\'ar components}\label{sec-appKol}

With the above discussions, we can start to prove our theorems.
\begin{proof}[Proof of Theorem \ref{t-approx}]
By Proposition \ref{p-inf}, we know 
$$\inf_{v} \hvol_{(X,D),o}(v)=\inf_{\fa}\lct^n(\fa)\cdot\mult(\fa).$$
By the above construction in Lemma \ref{l-model1} and \ref{l-kollar}, for any $\fm$-primary ideal $\fa$, we know that there exists a Koll\'ar component $S$, 
such that 
$$\hvol(\ord_S)\le \lct^n(\fa)\cdot\mult(\fa). $$  

Let $\{\fa_k\}_{k\in \Phi}$ be the associated graded family of valuation ideals induced by $v^{\rm m}$ where $\Phi\subset \bR$ is the value semigroup. For each $\fa_k$ ($k\in \Phi$), we denote 
$$c_k:=\lct(X,D; \fa_k).$$
Let $\mu_k\colon Y_k\to X$ be a dlt modification of $(X,D; c_k\cdot \fa_k)$ and $E_k$ the exceptional divisor of $Y_k$ over $X$. Assume the model we obtain from Lemma \ref{l-kollar} is $Y'_k$ with the Koll\'ar component $S_k$.
 
We consider the valuation 
$$v_k:=\frac{c_k\cdot k}{A_{(X,D)}(S_k)}\ord_{S_k}.$$ 
Note that $A_{(X,D)}(v_k)=c_k\cdot k$ is uniformly bounded: 
$$c_k\cdot k=\lct (X,D; \frac{1}{k}\fa_k)=\inf_{v'}\frac{A_{(X,D)}(v')}{\frac{1}{k}v'(\fa_k)}\le A_{(X,D)}(v^{\rm m}) <\infty.$$
So by the Izumi type estimate in \cite[Theorem 1.2]{Li15a}, we know that 
$$v_k(\fm)\ord_{o}\le v_k \le cA_{(X,D)}(v_k) \cdot \ord_{o}\le c'\cdot \ord_o,$$
for some positive constant $c, c'$ and all $k$. By \cite[Theorem 1.1]{Li15a} and the fact that $\hvol(v_k)$ is bounded from above, we know that $v_k(\fm)$ is bounded from below. In particular, by the compactness result \cite[Proposition 5.9]{JM12} and Proposition \ref{p-sequcom}, we know that there is an infinite sequence $\{v_{k_i}\}_{k_i\in \Phi}$ with $k_i\to +\infty$ which has a limit in $\Val_{X,o}$.
which we denote by
$$v'=\lim_{i\to \infty} v_{k_i}.$$ 
Then we know that 
$$A_{(X,D)}(v')\le \liminf_{i\to \infty} A_{(X,D)}(v_{k_i})=\liminf_{i\to \infty} c_{k_i}\cdot {k_i} \le A_{(X,D)}(v^{\rm m})$$ as $A_{(X,D)}$ is lower semicontinuous (see \cite[Lemma 5.7]{JM12}). We claim for any $f$, we have
$$v'(f)\ge v^{\rm m}(f).$$
Assuming this is true, then we $\vol(v')\le \vol(v^{\rm m})$, which then implies $\hvol(v')\le \hvol(v^{\rm m})$. Because $v^{\rm m}$ is a minimizer of $\hvol$, by Proposition \ref{p-valuation}, we must have
$v'=v^{\rm m}$.

To verify the claim, we pick any $f\in R$ and let $v^{\rm m}(f)= p$. For a fixed $k_j$, choose $l$ such that 
$$(l-1)p< k_j\le lp .$$ Let $k=k_j$ in the previous construction. Then we have: 
\begin{eqnarray*}
 v^{\rm m}(f)= p&\Longrightarrow & v^{\rm m}(f^l)= pl,\\
&\Longrightarrow & f^l\in \fa_{pl},\\
&\Longrightarrow & f^l\in \fa_{k_j},\\
&\Longrightarrow &l\cdot  \ord_{E_i}(f)\ge m_{k_j,i} \mbox{\ \  for any $i$},\\
&\Longrightarrow& l\cdot  \ord_{S_{k_j}}(f)\ge A_{(X,D)}(S_{k_j})\cdot \frac{1}{c_{k_j}},\\
&\Longrightarrow&v_{k_j}(f)\ge \frac{k_j}{l}>p-\frac{p}{l}.
\end{eqnarray*}
The fourth arrow is because if $f^l\in \fa_{k_j}$, then $f^l$ vanishes along $m_{k_j,i}G_{k_j,i}$;  and
the  fifth arrow is because  that
$$K_{Y_{k_j}}+\mu_{k_j*}^{-1}D+E_{k_j}\sim_{\mathbb{Q},X} c_{k_j}\cdot \sum m_{k_j,i}G_{k_j,i},$$
and the pull back of $K_{Y_{k_j}}+\mu_{k_j*}^{-1}D+E_{k_j}$ is larger than the one from 
$$K_{Y'_{k_j}}+\mu_{k_j*}^{'-1}D+S_{k_j}\sim_{\mathbb{Q},X} A_{(X,D)}(S_{k_j})S_{k_j}.$$
Thus $v'(f)=\lim v_{k_j}(f)\ge p= v^{\rm m}(f)$.
\end{proof}
\begin{prop}\label{p-sequcom} Let $o\in (X,D)$ be a klt singularity. Let $a$ and $b$ be two positive numbers. Then the subset $K_{a,b}$ of $\Val_{X,o}$ which consists of all valuations with 
$$a\le v(\fm) \qquad \mbox{and} \qquad A_{(X,D)}(v)\le b$$
 is sequential compact. 
\end{prop}
\begin{proof} Let $\{v_i\}$ be a sequence contained in $K$. Let $\{\fa_{i,k}\}$ be its associated graded sequence of valuative ideals for $k\in \Phi_i$. We can find a countably generated field $F\subset \mathbb{C}$, such that $R={\rm Spec}(R_F)\times_F \mathbb{C}$ for some finitely generated $F$-algebra $R_F$ and $D$, $o$ are defined over $F$. Furthermore, we can assume for each pair $(i,k)$, ${\fa_{i,k}}=({\fa_{i,k}})_F\times_F {\mathbb{C}}$, for some ideal $({\fa_{i,k}})_F\subset R_F$. Denote by $X_F:={\rm Spec}(R_F)$ and $D_F$ the divisor of $D$ descending on $X_F$.

Now let $(v_i)_F$ be the restriction of $v_i$ on $R_F$. By our definition, we know that
$$\fa_{i,k}=\{f\in R_F \ | (v_i)_F(f)\ge k\},$$
and  $(v_i)_F\in (K_{a,b})_F$ where $(K_{a,b})_F$ is defined for all $v\in \Val_{X_F,o}$ with  $a\le v(\fm_F)$ and $A_{X_F,D_F}(v)\le b$. By \cite[Theorem 1.1]{HLP12}, $\Val_{X_F,o}$ has the same topology as a set of some Euclidean space, thus $(K_{a,b})_F$ is sequential compact as it is compact by \cite[Proposition 5.9]{JM12}. Therefore after passing through a subsequence, $(v_i)_F$ has a limit $(v_{\infty})_F$, which can be extended to a valuation $v_{\infty}:=(v_{\infty})_F\otimes \mathbb{C}$. In fact, $v_{\infty}$ is defined as follows: for any $f\in R$ it can be written $f=\sum^m_{j=1}f_j\otimes_F h_j$ such that $0\neq f_j\in R$ and $h_1,...,h_m\in \mathbb{C}$ are linearly independent over $F$, then 
$$v_{\infty}(f)=\min^m_{j=1} \ (v_{\infty})_F(f_j).$$ We claim $v_i=(v_i|_{R_F})\otimes_F\mathbb{C}$. In fact, for any $f$, if $v_i(f)=k$, then $f\in \fa_{i,k}=(\fa_{i,k})_F\otimes_F \mathbb{C}$, thus $(v_i|_{R_F})\otimes_F\mathbb{C}(f)=k$. 

To see that for any $f$, $v_{\infty}(f)=\lim v_i(f)$, we know for some $j$, 
$$v_{\infty}(f)=(v_{\infty})_F(f_j)=\lim_{i}(v_i|_{R_F})(f_j)\ge \limsup_i v_i(f). $$
For another direction,  if we have a subsequence of $i$, such that $\lim_i v_i(f)<v_{\infty}(f)$, after passing to a subsequence again, we can find a $j$,  such that  $$\lim_iv_i(f)=\lim_iv_i(f_j)=\lim_i (v_{\infty})_F(f_j)\ge v_{\infty}(f), $$
a contradiction. 
\end{proof}
\begin{rem}
A referee pointed out that the sequential compactness of Berkovich space was studied in \cite{Poi13}. The above result could also be derived from this work.
\end{rem}

For a general klt singularity $(X,o)$, the minimum is not always achieved by a Koll\'ar component (see \cite{Blu16b, LX17}). Thus we have to take a limiting process.  However, if the minimizer $v$  is divisorial, then it should always yield a Koll\'ar component. First we have the following result inspired by the work in \cite{Blu16a} (we note that it is also independently obtained in \cite{Blu16b}).

\begin{lem}\label{l-lcplace}
If $\ord_E\in \Val_{X,o}$ minimizes $\hvol_{(X,D)}$, then the Rees algebra associated to $\ord_E$ is finitely generated. 
\end{lem}
\begin{proof}If we let $\{\fa_{\bullet}\}$ be the graded valuative ideas  associated to $\ord_E$, then we know that 
\begin{eqnarray*}
\hvol(\ord_{E})&= &\lim_{k\to \infty}A_{(X,D)}(\ord_E)^n \cdot \frac{\mult(\fa_k)}{k^n}\\
 &\ge& \lim_{k\to \infty}\lct (X,D;\fa_k)^n \cdot \mult(\fa_k)\\
 & \ge&\hvol(\ord_{E})
\end{eqnarray*}
by Proposition \ref{p-inf} and our assumption that $\ord_E$ is a minimizer of $\hvol_{(X,D)}$. So we conclude that (see \cite{Mus02})
$$\lct(X,D;\fa_{\bullet}):=\lim_{k\to \infty} k\cdot \lct (X,D;\fa_k)=A_{(X,D)}(\ord_E),$$ 
which we denote by $c$. Therefore,  we can choose $\epsilon$ sufficiently small,  such that the discrepancy $a(E;X,D+(1-\epsilon)c\cdot \fa_{\bullet})\in (-1,0),$

On the other hand, we know 
$$\lct(X,D;\fa_{\bullet})=\lim_{m\to \infty} m\cdot \lct(X,D;\fa_{m}).$$ So for sufficiently large $m$, we know that for all $G$, the discrepancy
$$a(G;X, D+\frac{1}{m}(1-\epsilon)c\cdot\fa_m)>-1.$$ 
We also have $a(E;X, D+\frac{1}{m}(1-\epsilon)c\cdot\fa_m)<0$.
Then similar to the discussion in \ref{d-dlt}, we can find a $\mathbb{Q}$-divisor $\Delta$, such that
$(X, D+\frac{1}{m}(1-\epsilon)c \cdot \Delta)$ is klt and $a(E;X, D+\frac{1}{m}(1-\epsilon)c\cdot \Delta)<0$. As a consequence we can apply \cite{BCHM10} to obtain a model $\mu\colon Y\to X$ such that ${\rm Ex}(\mu)=E$ and $-E$ is $\mu$-ample, which implies the finite generation. 
\end{proof}

\begin{proof}[Proof of Theorem \ref{t-divisor}] Applying Lemma \ref{l-lcplace},  the assumption in Case 1 which says $v$ is a divisorial valuation implies the assumption in Case 2, thus we only need to treat the Case 2.

\bigskip

By the proof of Proposition \ref{p-inf}, 
\begin{eqnarray*}
A_{(X,D)}(v)^n \cdot \frac{\mult(\fa_k)}{k^n}&\ge & \left(\frac{A_{(X,D)}(v)}{v(\fa_k)}\right)^n \cdot \mult(\fa_k)\ge \lct^n(X,D; \fa_k)\cdot \mult(\fa_k).
\end{eqnarray*}
By the finite generation assumption, we know that $\fa_{kl}=\fa^l_{k}$ for sufficiently divisible $k$ and any $l$.
So replace $k$ by $kl$ in the above display and let $l\to +\infty$, we know that
$$\hvol_{(X,D),o}(v)\ge \lct^n(X,D; \fa_k)\cdot \mult(\fa_k)\ge  \hvol_{(X,D),o}(v).$$
Take $\mu\colon Y\to X$ to be the dlt modification of $(X,D+\lct(X,D, \fa_k)\cdot \fa_k)$ as given in Lemma \ref{l-gooddlt}. The above discussion then implies that
$$ \lct^n(X,D; \fa_k)\cdot \mult(\fa_k)=  \hvol_{(X,D),o}(v)=\vol_{(X,D),o}(Y).$$

Moreover, it follows from Proposition \ref{p-special}, that we can choose a model $W\to Y$ and  running MMP to obtain $W\dasharrow Y'$, such that $\mu'\colon Y'\to X$ gives a Koll\'ar component $S$ with $a(S; Y,\mu^{-1}D_*+E)=-1$. We only need to show that if $Y'$ and $Y$ are not isomorphic in codimension 1, then 
$$\vol_{(X,D),o}(Y')<\vol_{(X,D),o}(Y).$$

This is the the local analog of the argument in \cite[Proposition 5]{LX14}. We give the details for the reader's convenience. 

Let $\pi\colon Y\to Y^{\rm c}$ be the canonical model of $-K_Y-{\mu}_*^{-1}D-E$ over $X$, which exists because 
$$-\epsilon (K_Y+{\mu}_*^{-1}D+E)\sim_{\mathbb{Q},X}K_Y+{\mu}_*^{-1}(D+c\cdot\frac{1}{l}\sum D_j)+E-\epsilon \sum_i A_{(X,D)}(G_i)G_i$$
is a klt pair for $\epsilon$ sufficiently small. The assumption that $Y'$ and $Y$ are not isomorphic in codimension 1 implies $Y^{\rm c}\neq Y$.

Take $p\colon \hat{Y}\to Y$ and $q\colon \hat{Y}\to Y'$ a common log resolution, and write
$$p^*(K_Y+{\mu}_*^{-1}D+E)=q^*(K_{Y'}+{\mu'}_*^{-1}D+S)+G.$$
 By negativity lemma (cf. \cite[3.39]{KM98}), we conclude that $G\ge 0$. Since 
 $$K_Y+{\mu}_*^{-1}D+E\sim_{\mathbb{Q},X}\sum_i A_{(X,D)}(G_i)G_i $$
 and
  $$K_{Y'}+{\mu'}_*^{-1}D+S\sim_{\mathbb{Q},X} A_{(X,D)}(S)\cdot S,$$
 we know that 
 $$p^*(\sum_iA_{(X,D)}(G_i)G_i)=q^*(A_{(X,D)}(S)\cdot S)+G.$$
 
  For $0\le \lambda \le 1$, let 
  $$L_{\lambda}=q^*(A_{(X,D)}(S)\cdot S)+\lambda G=\sum_{i} b_i(\lambda) F_i,$$
  where $F_i$ runs over all divisor supports on $\hat{Y}_{o}:=\hat{Y}\times_X\{o\}$, and $-L_{\lambda}|_{\hat{Y}_o}$ is nef.
  Define 
$$f(\lambda) = \sum_i b_i(\lambda)(-L_{\lambda}|_{F_i})^{n-1},$$
thus $f(\lambda)$ is non-decreasing as $G\ge 0$. By Lemma \ref{l-inter} and \ref{l-inter2}, we know that
$$f(1)=\vol_{(X,D),o}(Y)\qquad\mbox{and}\qquad f(0)=\vol_{(X,D),o}(Y').$$

Since $Y\dasharrow Y'$ are not isomorphic incodimension 1, it must contract some component $G_1$ of $E$, and the coefficient of $G_1$ in $G$ is
$$a:=A_{(Y',{\mu'}_*^{-1}D+S)}(G_1)>0.
$$
Then
\begin{eqnarray*}
\frac{df (\lambda)}{d \lambda}|_{\lambda=1}&=& n\cdot G\cdot \big(-p^*(K_Y+{\mu}_*^{-1}D+E)\big)^{n-1}\\
&\ge&n \cdot aG_1 \cdot \big(-\pi_*(K_Y+{\mu}_*^{-1}D+E)\big)^{n-1}\\
&>&0.
\end{eqnarray*}
Thus $\vol_{(X,D),o}(Y')=f(0)<f(1)=\vol_{(X,D),o}(Y)$.
\end{proof}

With all these discussions, we also obtain the following result, which characterizes the equality condition in Proposition \ref{p-inf} and is a corresponding generalization of \cite[Theorem 1.4]{dFEM04} (see Remark \ref{r-sharp}) for smooth point. See \cite[9.6]{LazII} for more background. 

\begin{thm}\label{p-equality}
Let $(X,o)=(\Spec (R),\fm)$. Assume $(X,D)$ is a klt singularity for a $\mathbb{Q}$-divisor $D\ge 0$.
Then there exists an $\fm$-primary ideal $\fa$ that obtains the minimum of normalized volume, i.e. 
 $$\lct^n(X,D;\fa)\cdot \mult (\fa)=\inf_{v\in \Val_{X,o}}\hvol(v),$$
 if and only if there exists a Koll\'ar component $S$ that satisfies the following two conditions:
 \begin{enumerate}
 \item[(1)] $\ord_S$ computes both $\lct(X,D; \fa)$ and $\inf_{v\in \Val_{X,o}}\hvol(v)$.
 \item[(2)] There exists a positive integer $k$ such that 
the only associated Rees valuation of $\fa^k$ is $\ord_S$.  
\end{enumerate}
\end{thm}
Later we will verify Theorem \ref{t-main2} which says such a minimizing Koll\'ar component $S$ is unique. 
\begin{proof}By the argument in Theorem \ref{t-divisor}, we see that 
$$\lct^n(X,D;\fa)\cdot \mult (\fa)$$
reaches the minimum of $\hvol_{(X,D),o}$ if and only if there is a dlt modification $\mu: Y\to X$ of 
$$(X,D; c\cdot \fa) \qquad \mbox {where $c=\lct(X,D; \fa)$} $$
that only extracts a Koll\'ar component $S$ of $(X,D)$ such that $\ord_S$ is a minimizer of $\hvol_{(X,D),o}$. 

Now we fix such an ideal $\fa$ and Koll\'{a}r component $S$. 
Assume that $\mu^*\fa$ has vanishing order $m$ along $S$. Since $S$ is $\mathbb{Q}$-Cartier, we can choose a positive integer $k$ such that $mk S$ is Cartier. We claim that
$$ \mu^*(\fa^k)=\mathcal{O}_Y(-mkS).$$
Granted this for now, then we know that $Y$ coincides with the normalized blow up $X^{+}\to X$ of $\fa^k$, i.e., $S$ is the only associated Rees valuation for $\fa^k$.

 To verify the claim, since $-mkS$ is Cartier, we know that 
 $$ \mu^*(\fa^k)=\fc\cdot \mathcal{O}_Y(-mkS)\qquad \mbox{for some ideal $\fc\subset \mathcal{O}_Y$},$$
 and we aim to show that $\fc$ is indeed trivial. If not, we take a normalized blow up $\phi\colon Y^+\to Y$ of $\fc$, so $\phi^*{\fc}=\mathcal{O}_{Y^+}(-E)$ for some effective Cartier divisor $E$. Since $-S$ is ample over $X$, we can  choose $l$ sufficiently big, such that 
 $$-D:=-\phi^*(mklS)-E$$ on $Y^+$ is ample over $X$. 
 
 Since 
 $$(\mu\circ \phi)^*\fa^{kl}=\mathcal{O}_{Y^+}(-\phi^*(mklS)-lE)\subset \mathcal{O}_{Y^+}(-\phi^*(mklS)-E),$$
 we know that 
 \begin{eqnarray*}
 \mult(\fa^{kl})&\ge & \vol^{F}_{o}(-\phi^*(mklS)-E)\\
  &= & \vol^{F}_o(-D)\\
  &=& mkl (-D|_{\phi^*S})^{n-1}+ (-D|_{E})^{n-1}\\
  &>&  mkl (mkl(-S)|_{S})^{n-1}\\
  &=&(mkl)^n\vol(\ord_S).
\end{eqnarray*}
Since $ \lct(X,D;\fa)=\frac{1}{m}\cdot A_{(X,D)}(\ord_S)$, we can easily see the above inequality is  contradictory to the assumption that 
$$\lct^n(X,D;\fa)\cdot \mult (\fa)=\hvol(\ord_S).$$
Here the inequality in the fourth row comes from a similar but easier calculation as in the proof of of Theorem \ref{t-divisor}. 
\bigskip

For the converse direction, we assume conditions (1)-(2) hold. 
We assume that $\ord_S(\fa)=m$ and that for some integer $k$ the only associated Rees valuation of $\fa^k$ is $\ord_{S}$, i.e., the normalized blow up of $\fa^k$, denoted by $\mu\colon X^{+}\to X$,  has the property that $\mu^*(\fa^k)=\mathcal{O}_{X^+}(-mkS)$. Then the valuative ideal 
\begin{equation}\label{eq-akint}
\fa_{mkl}(\ord_S)=\{f\in R; \ord_S(f)\ge mkl\}=\mu_*(\mu^*(\fa^k)^l)=\overline{\fa^{kl}},
\end{equation}
where $\overline{\fa^{kl}}$ means the integral closure of $\fa^{kl}$. By assumption, $\lct(X,D;\fa)=\frac{A_{(X,D)}(S)}{m}.$ 
We claim that 
$$\mult(\fa^k)=\lim_{k\to +\infty}\frac{n!\cdot l_R(R/\overline{\fa^{kl}})}{l^n} ,$$
and this together with \eqref{eq-akint} implies that 
\begin{eqnarray*}
\lct(X, D; \fa)^n\cdot \mult(\fa)&=&\lct(X,D;\fa^k)^n\cdot \mult(\fa^k)\\
&=&\frac{A_{(X,D)}(S)^n}{m^n}\lim_{l\rightarrow+\infty}\frac{n!\cdot l_R(R/\fa_{ml}(\ord_S)}{l^n}\\
&=&\hvol(\ord_S)=\inf_v\hvol(v).
\end{eqnarray*}

To verify the claim,  if we denote by $\mathcal{J}(\fa^{kl})=\mathcal{J}(X,D;\fa^{kl})$ the multiplier ideal, then we know that 
$$\mult(\fa^k)=\lim_{k\to +\infty}\frac{n!\cdot l_R(R/\mathcal{J}(\fa^{kl}) )}{l^n} ,$$
by the local Skoda Theorem \cite[9.6.39]{LazII}. On the other hand, since $(X,D)$ is klt, we have 
$$\fa^{kl}\subseteq \overline{\fa^{kl}} \subseteq \mathcal{J}(\fa^{kl}). $$
Thus we have 
\begin{eqnarray*}
\mult(\fa^k)&=&\lim_{l\rightarrow+\infty}\frac{n!\cdot l_R(R/\fa^{kl})}{l^n}\ge \lim_{l\rightarrow+\infty}\frac{n!\cdot l_R(R/\overline{\fa^{kl}})}{l^n}\\
&\ge& \lim_{l\rightarrow+\infty}\frac{n!\cdot l_R(R/\cJ(\fa^{kl}))}{l^n}=\mult(\fa^k).
\end{eqnarray*}
Thus the inequalities have to be identities and we are done.
\end{proof}
\begin{rem}\label{r-sharp}
In the proof, we indeed showed that if $\fa$ has the minimal normalized multiplicity and $S$ is  a Koll\'{a}r component such that $\ord_S(\fa)=m$ as in the statement of the above theorem, then for any $k$ such that $mkS$ is Cartier on $Y$,  the integral closure $\overline{\fa^k}$ coincides with the valuative ideal $\fa_{mk}$ of $\ord_{S}$ (see identity \eqref{eq-akint}). 
\end{rem}


\section{K-semistability implies the minimum}\label{s-min}
\subsection{Degeneration to initial ideals}\label{s-degin}
Let $(X, o)=({\rm Spec}(R), \fm)$ be an algebraic singularity such that $(X,D)$ is klt for a $\mathbb{Q}$-divisor $D\ge 0$. Given a Koll\'{a}r component $S$, we consider the associated degeneration $\cW^\circ/\bA^1$ of $X$ where $\cW^\circ$ is the underlying coarse moduli space of $\mathfrak{W}^\circ=\mathfrak{W}\backslash \mathfrak{Y}_0$ defined in Section \ref{ss-deformation}. We follow the notation in Section \ref{ss-deformation} and also denote by $v_0$ the valuation $\ord_S$. 

Suppose $\fb$ is an $\fm$-primary ideal on $X$. We will describe explicitly a way of obtaining an ideal $\mathfrak{B}$ on $\cW$ such that $\mathfrak{B}\otimes \cO_{X\times\bC^*}$ is the pull back of $\fb$ and $\mathfrak{B}\otimes \cO_C\cong \bin(\fb)$ by considering the closure of $\fb\times \mathbb{C}^*$ on $\cW$. 
For this purpose we consider the extended Rees algebra associated to the Koll\'{a}r component (see \cite[6.5]{Eis94}):
\[
\mathcal{R}'=\bigoplus_{k\in \bZ}\cR'_k:=\bigoplus_{k\in \bZ} \fa_k t^{-k}\subset R[t, t^{-1}],
\]
where $\fa_k=\fa_k(\ord_S)$. 
Notice that if $k\le 0$, then $\fa_k=R$. It is well known that the following identification holds true (recall that $R^*$ was defined in \eqref{eq-R*}):
\[
\mathcal{R}'\otimes_{\mathbb{C}[t]}\mathbb{C}[t,t^{-1}]\cong R[t, t^{-1}], \quad \mathcal{R}'\otimes_{\mathbb{C}[t]}\mathbb{C}[t]/(t)\cong \bigoplus_{k=0}^{+\infty}(\fa_k/\fa_{k+1})t^{-k}\cong R^*.
\]
Geometrically this exactly means $\cW^\circ=\Spec (\mathcal{R}')$ and
\[
\cW^\circ\times_{\bA^1}(\bA^1\setminus\{0\})=X\times (\bA^1\setminus\{0\}), \quad \cW^\circ\times_{\mathbb{A}^1}\{0\}=C.
\]
Notice that there is a natural $\bG_m$-action on $\cW^\circ$ given by the $\bZ$-grading.

For any $f\in R$, supposing $v_0(f)=k$ then we define 
$$\tilde{f}=t^{-k}f\in \fa_k t^{-k}\subset \cR',$$ and denote 
$$\bin(f)=[f]=[f]_{\fa_{k+1}}\in \fa_k/\fa_{k+1}=R^*_k,$$
where we use $[f]_{\fa}$ to denote the image of $f$ in $R/\fa$. 
Then we define the ideal $\kB$ to be the ideal in $\cR'$ generated by $\{\tilde{f}; f\in \fb\}$, and $\bin(\fb)$ the ideal of $R^*$ generated by 
$\{ \bin (f); f\in \fb \}$.
The first two items of the following lemma is similar to (but not the same as) \cite[Theorem 15.17]{Eis94} and should be well known to experts. Notice that here we degenerate both the ambient space and the ideal. A version of the equality \eqref{eqdim} was proved in \cite[Proposition 4.3]{Li15b}.
\begin{lem}[]
\begin{enumerate}
\item With the above notations, there are the identities:
\begin{equation*}\label{eqflat}
\left(\cR'/\kB\right) \otimes_{\mathbb{C}[t]}\mathbb{C}[t,t^{-1}]\cong (R/\fb)[t, t^{-1}], \quad \left(\cR'/\kB\right)\otimes_{k[t]}k[t]/(t)\cong R^*/\bin(\fb).
\end{equation*}
\item 
The $\mathbb{C}[t]$-algebra 
$\cR'/\kB$ is free and thus flat as a $\mathbb{C}[t]$-module. In particular, we have the identity of dimensions:
\begin{equation}\label{eqdim}
\dim_{\bC} \left(R/\fb\right)=\dim_{\bC}  \left(R^*/\bin(\fb)\right).
\end{equation}
\item
$\bin(\fb)$ is an $\fm_0$-primary homogeneous ideal, where $\fm_0=\sum_{k>0}R^*_k$.
\end{enumerate}
\end{lem}
\begin{proof}

The statement (1) follows easily from the definition.

\bigskip

Next we prove (2). Denote by $\fc_k=R^*_k\cap \bin(\fb)$ the $k$-th homogeneous piece of $\bin(\fb)$.
We fix a basis $\left\{\bin(f^{(k)}_i); 1\le i\le d_k \right\}$ of $R^*_k/\fc_k$. We want to show that 
\[
\cA':=\left\{\left.\left[\widetilde{f^{(k)}_i}\right]=\left[f^{(k)}_i\right]_{\kB}\; \right|\; 1\le i\le d_k \right\}\subset \cR'/\kB
\]
is a $\mathbb{C}[t]$-basis of $\cR'/\kB$. 

We first verify that $\cA'$ is a linearly independent set. To prove this, we just need to show that $\cA'$ is a $\mathbb{C}[t, t^{-1}]$-linearly independent subset of $(R/\fb)[t,t^{-1}]$. 
It is then enough to show that 
\begin{equation}\label{eqcA}
\cA:=\left\{[f^{(k)}_i]=[f^{(k)}_i]_\fb \; |\; 1\le i\le d_k\right\}\subset R/\fb.
\end{equation}
is $\mathbb{C}$-linearly independent, which can be verified directly as in \cite[Proposition 4.3]{Li15b}. See also \cite[Proposition 15.3]{Eis94}.

So we just need to show that $\cA'$ spans $\cR'/\kB$. Equivalently, we need to show that for any $f\in R$, $[\tilde{f}]=[\tilde{f}]_{\kB} \in \cR'/\kB$ is in the $\mathbb{C}[t]$-span of $\cA'$. 
This can be shown again with the help of $\cA$ in \eqref{eqcA}, that is, it is enough to prove that
$\cA$ spans $R/\fb$ as $\mathbb{C}$-linear space. Indeed, assuming the latter, for any $f\in R$, there exists a linear combination $g=\sum_{i,k}c_{ik} f^{(k)}_i$ such that 
$f-g=:h\in \fb$. If $m=v_0(f)$, then 
\[
\tilde{f}=t^{-m}f=\sum_{i,k}c_{ik}t^{-m}f^{(k)}_i+t^{-m}h
\]
Because $t^{-m}h\in \kB$, the above indeed implies $[\tilde{f}]$ is in the $\mathbb{C}[t]$-span of $\cA'$.

To prove that $\cA$ indeed $\mathbb{C}$-spans $R/\fb$, we first claim that the following set is finite:  
\[
\{v_0(g)\; |\; g\in R-\fb\}.
\]
Indeed because $\fb$ is $\fm$-primary, there exists $N>0$ such that $\fm^N\subseteq \fb\subseteq \fm$. So $R-\fb\subseteq R-\fm^N$. Now the claim follows from the fact that for  any element $f\in \fm^N$,  
$$v_0(f)\le c \cdot A(v_0)\cdot N$$
by Izumi's theorem,  where $c$ is a uniform constant not depending on $f$.

If there is $[f]\neq 0\in R/\fb$ that is not in the span of $\cA$, then we can choose a maximal $k=v_0(f)$ such that this happens. There are two cases:
\begin{enumerate}
\item If $\bin(f)\in R^*_k\setminus \fc_k$, then because $\bin(f^{(k)}_i)$ is a basis of $R^*_k/\fc_k$, there exists $t_j\in \bC$ such that $\bin(f)-\sum_{j=1}^{d_k} t_j \bin(f^{(k)}_j)=\bin(g) \in \fc_k$ for some $g\in\fb$. So we get:
\[
v_0\left(f-\sum_{j=1}^{d_k} t_j f^{(k)}_j-g\right)>k.
\]
 By maximality 
of $k$, $[f-\sum_{j=1}^{d_k} t_j f^{(k)}_j-g]=[f]-\sum_{j=1}^{d_k} t_j [f^{(k)}_j]$ and hence $[f]$ is in the span of $\cA$. Contradiction.
\item If $\bin(f)\in \fc_k=\bin(\fb)\cap R^*_k$. Then $\bin(f)=\bin(g)$ for some $g\in \fb$. So $v_0(f-g)>k$ and hence $[f-g]$ is in the span 
of $\cA$ by the maximal property of $k$. But then $[f]=[f-g]+[g]=[f-g]$ is in the span of $\cA$. Contradiction.
\end{enumerate}

\bigskip
To prove part 3 of the Lemma, we need to show that there exists $N\in \bZ_{>0}$ such that $\fm_0^{N}\subseteq \bin(\fb)\subseteq \fm_0$. Because $\fb$ is $\fm$-primary, there exists $N_1\in \bZ_{>0}$ such that $\fm^{N_1}\subseteq \fb \subseteq \fm$. By Izumi's theorem, 
there exists $l\in\bZ_{>0}$ such that $\fa_{l m}\subseteq \fm^m$ for any $m\in \bZ_{>0}$. By letting $N=l N_1$, it is easy to see that $\fm^{N}_0\subseteq \bin(\fb)\subseteq\fm_0$.

\end{proof}

\begin{lem}
If $\fb_\bullet=\{\fb_k\}$ is a graded family of ideals of $R$, then $\bin(\fb_\bullet):=\{\bin(\fb_k)\}$ is also a graded family
of ideals of $R^*$. 
\end{lem}
\begin{proof}
We just need to show that:
\[
\bin(\fb_k)\cdot \bin(\fb_l)\subseteq \bin(\fb_{k+l}).
\]
If $v_0(f)=k$ and $v_0(g)=l$, then $v_0(fg)=k+l$.
\[
\bin(f)\cdot \bin(g)=[f]_{\fa_{k+1}}\cdot [g]_{\fa_{l+1}}=[f g]_{\fa_{k+l+1}}=\bin(f\cdot g).
\]
\end{proof}

\begin{lem}\label{l-deg}
If $\fb_\bullet$ is a graded family of ideals, then 
\begin{equation}\label{eqdeg}
\lct^n(\fb_\bullet)\cdot \mult(\fb_\bullet)\ge \lct^n(\bin(\fb_\bullet))\cdot \mult(\bin(\fb_\bullet)).
\end{equation}
\end{lem}
\begin{proof}
By the flatness of $\kB$ and the lower semicontinuity of log canonical thresholds, we have $\lct(\fb_k)\ge \lct(\bin(\fb_k))$. Therefore, by \eqref{eqdim}
\begin{eqnarray*}
\lct^n(\fb_k)\cdot l_R(R/\fb_k)&\ge& \lct^n(\bin(\fb_k))\cdot l_{R^*}(R^*/\bin(\fb_k)).
\end{eqnarray*}
Taking limits as $k\rightarrow+\infty$, we then get the inequality \eqref{eqdeg}.
\end{proof}

\subsection{Equivariant K-semistability and minimizer}\label{s-equiv}

In this section, we will take a detour to show the discussion in Section \ref{s-degin} can be used to study the equivariant K-semistability. Here for a $\mathbb{Q}$-Fano variety $(V,\Delta)$ with an action by an algebraic group $G$, we call it {\it $G$-equivariantly K-semistable (resp. Ding semistable)} if for any $G$-equivariant test configuration, its generalized Futaki (resp. Ding) invariant is non-negative.  Let $T=(\mathbb{C}^*)^r$ be a torus. First we improve the two approximating results to the equivariant case.
\begin{prop}\label{l-Tmini}
Let $(X,o)=({\rm Spec}R, \fm)$ and $D\ge 0$ a $\mathbb{Q}$-divisor, such that $o\in (X,D)$ is a klt singularity. Assume $o\in (X,D)$ admits a $T$-action.  Then we have 
\begin{equation}\label{eq-Tvol2mul}
\min_{v} \hvol_{(X,D),o}(v)=\inf_{\fa}\lct^n(X,D; \fa)\cdot\mult(\fa)=\inf_{S}\hvol_{(X,D),o}(\ord_{S}),
\end{equation}
where on the left hand side $v$ runs over all the valuations centered at $o$, and in the middle $\fa$ over all the $T$-equivariant $\fm$-primary ideals; and at the end, $S$ runs over all $T$-equivariant Koll\'ar components. 
\end{prop}
\begin{proof} 
Let $\fa_{\bullet}=\{\fa^k\}$ be a graded sequence for an $\fm$-primary ideal $\fa$. Assume $T\cong (\bC^*)^r$. Fixing a lexicographic order on $\bZ^r$, we can degenerate the ideal $\fa^k$ to its initial ideal $\bin(\fa^k)$.

Lemma \ref{l-deg} implies that for  $\fb_{\bullet}=\{\fb_k\}=_{\rm defn} \{ \bin(\fa^k)\}$ 
$$\lct^n(X,D; \fb_{\bullet})\cdot\mult(\fb_{\bullet})\le \lct^n(X,D; \fa_{\bullet})\cdot\mult(\fa_{\bullet}).$$ 
Since there is the identity:
$$\lct^n(X,D; \fb_{\bullet})\cdot\mult(\fb_{\bullet})=\lim_{m} \lct^n(X,D; \fb_{m})\cdot\mult(\fb_{m}),$$
we conclude the first inequality as a corollary of Proposition \ref{p-inf}. 

\medskip

For the second equality, we just need to show that the construction in Section \ref{s-lvmodel} can be established $T$-equivariantly. This is standard, which relies on two facts: first, we can always take an equivariant log resoltuion of $(X,D, \fa)$ (see \cite{Kol07}); second,  as $T$ is a connected group, for any curve $C$ in a $T$-variety and any $t\in T$, $t\cdot C$ will always be numerically equivalent to $C$; as the minimal model program only depends on the numerical class $[C]$, we know that any MMP sequence is automatically $T$-equivariant. 
Therefore, for any $T$-equivariant $\fm$-primary ideal $\fa$, we can find a $T$-equivariant dlt modication $Y\to X$ and then a $T$-invariant Koll\'ar component $S$, such that
$$\lct^n(X,D; \fa)\cdot\mult(\fa)\ge \vol(Y)\ge \hvol( \ord_{S}).$$
\end{proof}

In \cite{Li15b} (see also \cite{LL16}), it was proved that  the canonical valuation on the affine cone minimizes $\hvol_X$ implies $V$ is K-semistable. Conversly, if $V$ is K-semistable then the canonical valuation  minimizes $\hvol_X$ among all $\bC^*$-invariant valuations. The argument extends easily to the logarithmic case. Proposition \ref{l-Tmini} allows us to extend the minimization result to all valuations in $\Val_X$. (\cite{LL16} proved the same result, but under the the assumption that $V$ degenerates to a Fano with K\"ahler-Einstein metric.) 
For the reader's convenience, we sketch the argument from \cite{Li15b, LL16}.
\begin{thm}\label{thm-Ksemi} 
Let $(V,\Delta)$ be a projective log Fano variety and $o\in (X,D)$ is the affine cone over $(V,\Delta)$ induced by some ample Cartier divisor $L=-r^{-1}(K_V+\Delta)$. Then the canonical valuation $v_0$ obtained by blowing up the vertex minimizes $\hvol_{(X,D)}$ on $\Val_{X,o}$ if and only if $(V,\Delta)$ is log-K-semistable.
\end{thm}
\begin{proof}

First we assume that $(V, \Delta)$ is log-K-semistable and prove the volume minimizing property of $\ord_V$. By Proposition \ref{l-Tmini}, we only need to prove that for any $\bC^*$-invariant divisorial valuation $v$ over $(X, o)$, 
$$\hvol(v_0)\le \hvol(v).$$ 

Let $Y\rightarrow X$ be the blow-up at $o$ with the exceptional divisor still denoted by $V$. Denote by $\cI_V$ the ideal sheaf of $V\subset Y$ and define (see \cite[Lemma 4.2]{Li15b})
$$c_1:=c_1(\cI_V)=\min\left\{v(\phi); \phi\in \cI_V(U), U\cap {\rm center}_Y(U)\neq \emptyset\right\}.$$
Denote $R=\bigoplus_{k=0}^{+\infty}R_k=\bigoplus_{k=0}^{+\infty}H^0(V, k L)$ such that $X={\rm Spec}(R)$. 
On $R$, we define a graded filtration
$$\cF R^{(t)} = \bigoplus_{k=0}^{+\infty} \cF^{kt}R_k, \quad \text{ with }
\cF^xR_k:=\{f\in R_k; v(f)\ge x\}. $$
The volume of $\cF R^{(t)}$ is defined to be
$$\vol(\cF R^{(t)}) := \limsup_{m\to \infty} \frac{\dim_{\bC} \cF^{mt}R_m}{m^n/n!}.$$
 By \cite[(21) and (22)]{Li15b}, we get a formula for
$\vol(v)$:  
\begin{eqnarray*}
\vol(v)&=&\lim_{m\rightarrow+\infty} \frac{n!}{m^n}\dim_{\bC} R/\fa_m(v)\\
&=&\frac{L^{n-1}}{c_1^n}-\int^{+\infty}_{c_1}\vol\left(\cF R^{(t)}\right)\frac{dt}{t^{n+1}}\\
&=&-\int^{+\infty}_{c_1}\frac{d \vol\left(\cF R^{(t)}\right)}{t^n}.
\end{eqnarray*}
Then we consider the following function 
\begin{eqnarray*}
\Phi(\lambda, s)&=&\frac{L^{n-1}}{(\lambda c_1 s+(1-s))^n}-n \int^{+\infty}_{c_1}\vol\left(\cF R^{(t)}\right)\frac{\lambda s dt}{(1-s+\lambda st)^{n+1}}\\
&=&\int^{+\infty}_{c_1}\frac{-d \; \vol(\cF R^{(t)})}{((1-s)+\lambda st)^n}.
\end{eqnarray*}
$\Phi(\lambda, s)$ satisfies the following properties:
\begin{enumerate}
\item For any $\lambda \in (0, +\infty)$, we have:
\[
\Phi(\lambda, 1)=\vol(\lambda v)=\lambda^{-n} \vol(v), \quad \Phi(\lambda, 0)=\vol(v_0)=L^{n-1}.
\]
\item For any $\lambda \in (0, +\infty)$, $\Phi(\lambda, s)$ is continuous and convex with respect to $s\in [0, 1]$.
\item The directional derivative of $\Phi(\lambda, s)$ at $s=0$ is equal to:
\[
\Phi_s(\lambda, 0)=n \lambda L^{n-1} \left(\lambda^{-1}-c_1-\frac{1}{L^{n-1}}\int^{+\infty}_{c_1} \vol\left(\cF R^{(t)}\right)dt\right).
\]
\end{enumerate}
Let $\lambda_*=\frac{r}{A_{(X, D)}(v)}$. Note that $A_{(X,D)}(v_0)=r$. So by item 1, we have:
$$
\Phi(\lambda_*,1)=\frac{\hvol(v)}{r^n}, \quad \Phi(\lambda_*, 0)=L^{n-1}=\frac{\hvol(v_0)}{r^n}.
$$
By item 2, we just need to prove $\Phi_s(\lambda_*, 0)\ge 0$. 
Let $\bar{v}=v|_{\bC(V)}$ be the restriction of $v$ under the inclusion $\bC(V)\hookrightarrow \bC(X)$. 
It is known that $\bar{v}=b\cdot \ord_E$ where $b\ge 0$ by \cite[Proof of Lemma 4.1]{BHJ15} and $\ord_E$ is a divisorial valuation on $\bC(V)$. Moreover $v$ is the $\bC^*$-invariant extension of $\bar{v}$ to $\bC(X)$ (cf. \cite[Lemma 4.2]{BHJ15}, \cite[Appendix 4.2.1]{Li15b}):
\begin{equation}
v(f)=\min \{c_1k+\bar{v}(f_k); f=\sum_k f_k\in R \text{ with } f_k\neq 0\in R_k\}.
\end{equation}
If $\phi: \tilde{V}\rightarrow V$ is a model that contains $E$ as a divisor, then $v$ can also be obtained as a quasi-monomial valuation on the model $\tilde{Y}\rightarrow Y$ where $\tilde{Y}=Y\times_{V}\tilde{V}$ (see \cite[Definition 6.12]{Li15b}). Using this description, it is easy to show that:
\[
\lambda^{-1}_*-c_1=\frac{A_{(X, D)}(v)}{r}-c_1= \frac{A_{(V,\Delta)}(\bar{v})}{r}=\frac{b\cdot A_{(V,\Delta)}(E)}{r}.
\]
By change of variables we get:
\[
\int^{+\infty}_{c_1}\vol\left(\cF R^{(t)}\right)dt=\int^{+\infty}_0 \vol\left(\cF_{\bar{v}} R^{(t)}\right) dt,
\]
where
\[
\cF_{\bar{v}} R^{(t)} = \bigoplus_k H^0(V, L^{\otimes k}\otimes \fa_{kt}), \qquad \mbox{and \ \   }  \fa_{kt}=\{f\in \cO_V\ |\  \bar{v}(f)\ge kt\}.
\]
So we get the equalities:
\begin{eqnarray*}
\Phi_s(\lambda_*, 0)&=&n \lambda_* L^{n-1}\left(A_{(V,\Delta)}(\bar{v})-\frac{r}{L^{n-1}}\int^{+\infty}_{0} \vol\left(\cF_{\bar{v}}R^{(t)}\right)dt\right)\\
&=&n \lambda_* L^{n-1} b\left(A_{(V,\Delta)}(E)-\frac{r}{L^{n-1}}\int_0^{+\infty} \vol\left(\cF_{\ord_E}R^{(t)}\right)dt \right).
\end{eqnarray*}
By applying Fujita's result in \cite{Fuj15} (see also \cite{Fuj16, Li15b, LL16}), we get $\Phi_s(\lambda_*, 0)\ge 0$. 

Conversely, if $\ord_V$ is volume minimizing, then the above calculation shows that
\begin{equation}
A_{(V, \Delta)}(\ord_E)-\frac{r}{L^{n-1}}\int_0^{+\infty} \vol\left(\cF_{\ord_E}R^{(t)}\right)dt
\end{equation}
is non-negative for any divisorial valuation $\ord_E$ over $V$. By the valuative criterion for (log-)K-semistability in \cite{Fuj16, Li15b, LL16}, this implies $(V, \Delta)$ is indeed log-K-semistable.
\end{proof}
An alternative way to prove the first implication of Theorem \ref{thm-Ksemi} is using Proposition \ref{p-fanocone} and the arguments of \cite[Section 4.2]{LL16}.
With all the techniques we have, we can prove Theorem \ref{t-equiK}.
\begin{proof}[Proof of Theorem \ref{t-equiK}]  
 Let $(X,D)$ be the affine cone of $L=-r^{-1}(K_V+\Delta)$ over $(V,\Delta)$ for $r^{-1}$ being some sufficiently divisible positive integer. We consider the minimizing problem of the normalized local volume  at the $T$-equivariant singularity $o$ which is the vertex. We aim to show that if $(V, \Delta)$ is $T$-equivariantly log-K-semistable then $\ord_V$ minimizes $\hvol_{(X,D)}$. 
This then implies that $(V,\Delta)$ is log-K-semistable by Theorem \ref{thm-Ksemi} .

Following the proof of Proposition \ref{l-Tmini}, we assume that $T=(\bC^*)^r$ and fix a lexicographic order on $\bZ^r$. Then by taking initial ideals, we can always associate a graded sequence of $T$-equivariant ideals to a given primary ideal. On the other hand, 
$$\inf_{\fa} \lct^n(X,D; \fa)\cdot \mult(\fa)=\inf_{\fa_\bullet}\lct^n(X,D; \fa_\bullet)\cdot\mult(\fa_\bullet)=\min_{v\in \Val_{X,o}} \hvol(v).$$ 
So we can find a sequence of $T$-equivariant ideals $\{\fa_i\}$ such that 
$$\inf_i \lct^n(\fa_i)\cdot \mult(\fa_i)=\min_{v\in \Val_{X,o}}\hvol(v).$$
Using the equivariant resolution and running an MMP process as in Section \ref{sec-appKol}, we can find a  sequence of $T$-equivariant Koll\'ar components  $S_i$ such that 
$$\inf_{i} \hvol(\ord_{S_i})= \min_{v\in \Val_{X,o}} \hvol(v). $$
For any $T$-equivariant Koll\'ar component $S_i$, we consider $v=\ord_{S_i}\in \Val_{X,o}$. Denote its induced divisorial valuation on $V$ by $b\cdot \ord_E$.

Arguing as in the proof of Theorem \ref{thm-Ksemi}, in order to conclude $\hvol(v_*)\le \hvol(\ord_{S_i})$, we want to show that $\Phi_s(\lambda_*, 0)\ge 0$, where
$$\Phi_s(\lambda_*, 0)=n\lambda_* L^{n-1} b\left(A_{(V,\Delta)}(E)-\frac{r}{L^{n-1}}\int_0^{+\infty} \vol\left(\cF_{\ord_E}R^{(t)}\right)dt\right)$$
for the $T$-equivariant divisorial valuation $E$ over $(V,\Delta)$.

Now we use the assumption that $(V, \Delta)$ is $T$-equivariantly K-semistable. Following the argument in \cite{BBJ15, Fuj16}, we know that $(V, \Delta)$ is $T$-equivariantly Ding-semistable.
Indeed, for any special test configuration, the Futaki invariant is the same as the Ding invariant. Using the fact that $T$-equivariant MMP decreases Ding invariant by \cite{BBJ15, Fuj16}, we know this implies that the Ding invariant for any $T$-equvariant test configuration is nonnegative. Applying the argument in \cite{Fuj15} (see \cite{Li15b, Fuj16}), we conclude that $\Phi_s(\lambda_*, 0)\ge 0$ as wanted.

\end{proof}


\subsection{Proof of Theorem \ref{t-main}}
Let $(X, o)=({\rm Spec}(R), \fm)$ be an algebraic singularity such that $(X,D)$ is klt for a $\mathbb{Q}$-divisor $D\ge 0$. 
Let $S$ be a Koll\'{a}r component and $\Delta=\Delta_S$ be the different divisor defined by the adjunction $(K_Y+S+\mu_*^{-1}D)|_S=K_S+\Delta_S$ where $\mu \colon Y\rightarrow X$ is the extraction of $S$. We follow the notation in Section \ref{ss-deformation} and \ref{s-degin}. In this section, we will prove Theorem \ref{t-main} which states that if $(S, \Delta_S)$ is K-semistable, then $\ord_S$ minimizes $\hvol_X$ over $\Val_{X,o}$.
\begin{lem}\label{l-ksemi} 
Let $\fb_\bullet$ be a graded sequence of $\fm_0$-primary ideal whose reduced support is $o_C\in C$. 
If $(S,\Delta_S)$ is K-semistable, then we have
$$ \lct^n(\fb_\bullet)\cdot \mult(\fb_\bullet)\ge \hvol_{(C,C_D),o_C}(\ord_S).$$
\end{lem}
\begin{proof} 
Using the result in \cite{JM12}, we have 
\begin{eqnarray*}
\lct^n(\fb_\bullet)\cdot \mult(\fb_\bullet)&=&\lim_{k \to +\infty} \big(k\cdot \lct(\fb_k) \big)^n\cdot \frac {\mult(\fb_k)}{k^n}\\
&=&\lim_{k\to +\infty} \lct^n(\fb_k)\cdot \mult(\fb_k).
\end{eqnarray*}
By Proposition \ref{p-inf},  it suffices to show that  $\hvol_{(C,C_D),o_C}(\ord_S)$ is equal to
$$\min_{v} \hvol_{(C,C_D),o_C}(v)$$
for $v$  runs over valuations centered on $o_C$. 

It follows from  Theorem \ref{thm-Ksemi} that if we choose $d$ sufficiently divisible, such that $C^{(d)}=C(S,H)$ is constructed as the cone over $S$ with an ample Cartier divisor $H$ proportional to $-(K_{S}+\Delta_{S})$, then the canonical valuation $\ord_{S^{(d)}}$ is a minimizer of $\hvol_{(C^{(d)}, C^{(d)}_{1}+C^{(d)}_2)}$.   By Proposition \ref{p-cover}, this implies the same holds for $C$.
\end{proof}

\begin{prop}\label{p-cover}
With the above notations,  $\hvol_{({C}, C_D)}$ minimizes at $\ord_S$ if and only if $\hvol_{({C}^{(d)}, C^{(d)}_{1}+C^{(d)}_2)}$ minimizes at $\ord_{S^{(d)}}$. 
\end{prop}
\begin{proof}The degree $d$ cover $h\colon C\to C^{(d)}$ is a fiberwise map with respect to the cone structures and the Galois group $G=_{\rm defn}\mathbb{Z}/d$ is natural a subgroup of $\mathbb{C}^*$.
Let $E$ be a Koll\'ar component over $C^{(d)}$. By Lemma \ref{l-finite} we know $h^*(E)$ is a Koll\'ar component over $ C$, and it follows from Lemma \ref{l-finitevolume} (or \cite[Lemma 6.9]{Li15b}) that
$$d\cdot \hvol(\ord_E)=\hvol(h^*E). $$
 So if $\ord_S$ minimizes $\hvol_{(C,C_D)}$, then the corresponding  canonical valuation also minimizes $\hvol_{({C}^{(d)}, C^{(d)}_{1}+C^{(d)}_2)}$. 
 
 For the converse, let $E$ be a $T$-invariant Koll\'ar component over $C$. Since it is $G$-invariant, by Lemma \ref{l-finite} we know that it is a pull back of a Koll\'ar component $F$ over $C^{(d)}$. Assume that the canonical valuation minimizes $\hvol_{({C}^{(d)}, C^{(d)}_{1}+C^{(d)}_2)}$. Then over $C$, we see that $\hvol(\ord_S)$ is less than or equal to $\hvol(\ord_E)$ for any $T$-equivariant Koll\'ar component $E$. Therefore $\ord_S$ is a minimizer of $\hvol_{(C,C_{D})}$ by Proposition \ref{l-Tmini}. 
\end{proof}

Theorem \ref{t-main} is implied by Theorem \ref{thm-Ksemi} and the following proposition.
\begin{prop}\label{p-degvol}
Given any Koll\'ar component $S$ over $o\in (X,D)$, it induces a $\bC^*$-equivariant degeneration to an `orbifold' cone $o_C\in(C,C_D)$ with a Koll\'ar component $S_0\cong S$ which is the canonical valuation with respect to the orbifold cone structure, and we have 
$$\hvol_{(X,D),o}(\ord_S)=\hvol_{(C,C_D),o_C}(\ord_{S_0}) \qquad \mbox{and}\ \ \ \vol(o, X,D))\ge \vol(o_C, C, C_D). $$
\end{prop}
\begin{proof} 
We use same notations as in Section \ref{ss-deformation}. In particular, we denote by  $\cZ$ (resp. $\cW$) the coarse moduli space of $\mathfrak{Z}$ (resp. $\mathfrak{W}$). Let $\phi\colon \cZ\to X_{\mathbb{A}^1}(=X\times \mathbb{A}^1 )$ be the birational morphism and $S'_{\mathbb{A}^1}$ the birational transform of $ S_{\mathbb{A}^1} \subset Y_{\mathbb{A}^1}$ on $\cZ$. Write $aS'_{\mathbb{A}^1}\sim_{\mathbb{Q},\cW}K_{\cZ}+\phi^{-1}_*D_{\mathbb{A}^1}+S'_{\mathbb{A}^1}$. Restricting over a general fiber and taking the coarse moduli spaces, we obtain 
$$aS\sim_{\bQ,X} K_Y+S+\mu^{-1}_*(D),$$
then $a=A_{(X,D)}(S)$.  Similarly, over the central fiber, we get
$$aS_0\sim_{\bQ,C} K_{Y_0}+S_0+(\mu^{-1}_{0})_* C_D,$$
 where $\mu_0\colon Y_0\to C$  is the blow up of the vertex $o_C$ with the exceptional divisor $S_0\cong S$.  Thus $a=A_{C,C_D}(S_0)$.

We also know that 
$$\vol_{X,o}(\ord_S)=(-S|_S)^{n-1}=(-S_0|_{S_0})^{n-1}=\vol_{C,o_C}(\ord_{S_0}).$$

Combining all the above, we know that for any ideal $\fb$ on $X$, if we  let $\fb_{\bullet}=\{\fb^k\}$, then
\begin{eqnarray*} 
\hvol_{(X,D),o}(\ord_S) &= &\vol_{X,o}(\ord_S)\cdot A^n_{(X,D)}(S) \\
 & =  &\vol_{C,o_C}(\ord_{S_0})\cdot A^n_{(C,{C_{D}})}(S_0)\\
 &\le & \lct^n(\bin(\fb_\bullet))\cdot \mult(\bin(\fb_\bullet))\\
 &\le & \lct^n(\fb) \cdot  \mult (\fb),
\end{eqnarray*}
where  the last two inequalities follow from Lemma  \ref{l-ksemi}  and \ref{l-deg}.  
Thus we conclude that 
$$\hvol_{(X,D),o}(\ord_S)\le \inf_{\fb} \lct^n(\fb) \cdot  \mult (\fb)=\inf_v \hvol_{(X,D),o}(v),$$
where the second equality follows from Proposition \ref{p-inf}.
\end{proof}


\section{Uniqueness}\label{s-uni}

In this section, we will prove Theorem \ref{t-main2} on the uniqueness of the minimizers among all Koll\'ar components. There are two steps: first we prove this for cone singularities; then for a general singularity, we combine the deformation construction with some results from the minimal model program to essentially reduce it to the case of cone singularities.

\subsection{Case of cone singularity}\label{ss-ucone}

We first settle the case of cone singularities. It can be proved using Proposition \ref{p-Tequiv} and \cite[Theorem 3.4]{Li15b}. Here we give a different proof, which analyzes the geometry in more details. A similar argument in the global case appears in the proof of \cite[Theorem 3]{Liu16}, where a characterization of quotients of $\mathbb{P}^n$ was given as those achieving the maximal possible volumes among all K-semistable $\mathbb{Q}$-Fano varieties with only quotient singularities.

\bigskip

Let $(V,\Delta)$ be an $(n-1)$-dimensional log Fano variety and $-(K_V+\Delta)=r H$ for some $r\in \bQ$ and an ample Cartier divisor $H$. We assume $r\le n$. Let $X^0:=C(V, H)$ be the affine cone over the base $V$ with the vertex $o$ and let $X$ be the projective cone and $D$ be the cone divisor over $\Delta$ on $X$. 

Consider a Koll\'ar component $S$ over $o\in (X,D)$ with the extraction morphism $\mu\colon Y\to X$. Let 
$\mu_{\mathbb{A}^1}\colon Y_{\mathbb{A}^1}\rightarrow X_{\mathbb{A}^1}$
be the extraction of $S_{\mathbb{A}^1}$. We carry out the process of deformation to normal cones as in Section \ref{ss-deformation} with respect to $S$. Here $X$ is a projective variety instead of a local singularity, but the construction is exactly the same.  We denote by $\cZ$ (resp. $\cW$) the coarse moduli space of $\mathfrak{Z}$ (resp. $\mathfrak{W}$), so there are morphisms, $\psi_1: \cZ\to \cW$, $\phi_1\colon \cZ\to Y_{\mathbb{A}^1}$ and $\pi\colon \cW\to X_{\mathbb{A}^1}$. We denote by $\phi=\mu_{\mathbb{A}^1}\circ\phi_1.$
\begin{center}
\begin{figure}[h]
\centering
\includegraphics[width=8cm, height=5.5cm]{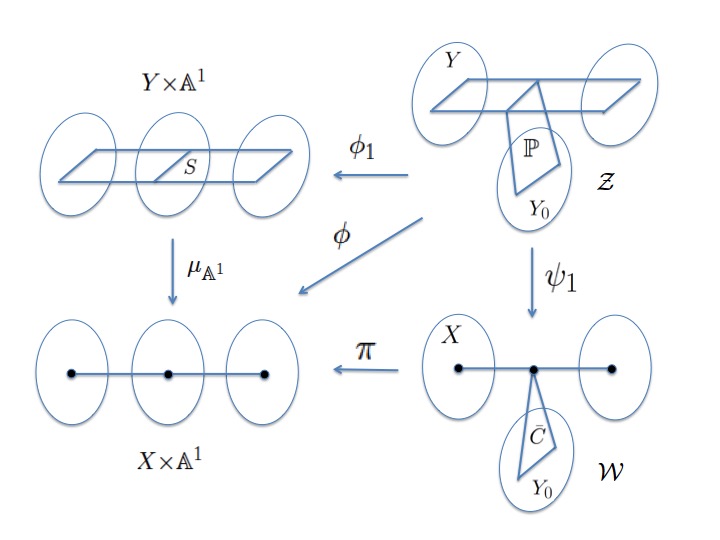}
\caption{Degeneration associated to a Koll\'{a}r component}
\label{fig-deg}
\end{figure}
\end{center}

Denote by $ \bP$ the irreducible exceptional divisor for $\phi_1$. We have the following equalities:
\begin{enumerate}
\item
$K_{Y_{\mathbb{A}^1}}+(\mu_{\mathbb{A}^1})_*^{-1}D_{\mathbb{A}^1}=\mu_{\mathbb{A}^1}^* (K_{X_{\mathbb{A}^1}}+D_{\mathbb{A}^1})+a  S_{\mathbb{A}^1} $ with $a=A_{(X,D)}(S)-1$;
\item
$K_{\cZ}+\phi_*^{-1}D_{\mathbb{A}^1}=\phi_1^*\big(K_{Y_{\mathbb{A}^1}}+(\mu_{\mathbb{A}^1})_*^{-1}D_{\mathbb{A}^1}\big)+\bP$;
\item
$K_{\cZ}+\phi_*^{-1}D_{\mathbb{A}^1}=\psi_1^* \big(K_{\cW}+(D_{\mathbb{A}^1})_\cW\big)+a  S'_{\mathbb{A}^1}$, where $(D_{\mathbb{A}^1})_\cW:=\psi_{1*}\phi_*^{-1}(D_{\mathbb{A}^1})$ and $S'_{\bA^1}=(\phi_1^{-1})_*(S_{\bA^1})$.
\end{enumerate}
The first two equalities imply:
\begin{eqnarray*}
K_{\cZ}+\phi_*^{-1}(D_{\mathbb{A}^1})&=&\phi_1^* \big(K_{Y_{\mathbb{A}^1}}+(\mu_{\mathbb{A}^1})_*^{-1}D_{\mathbb{A}^1}\big)+ \bP\\
 &=&\phi_1^*\mu_{\mathbb{A}^1}^* (K_{X_{\mathbb{A}^1}}+D_{\mathbb{A}^1})+a \phi_1^*S_{\mathbb{A}^1}+\bP\\
&=&\phi^*(K_{X_{\mathbb{A}^1}}+D_{\mathbb{A}^1})+a S'_{\mathbb{A}^1}+(a+1)\bP.
\end{eqnarray*}
So $A_{X_{\mathbb{A}^1},D_{\mathbb{A}^1}}(\bP)=a+2=A_{(X,D)}(S)+1$. This implies:
\[
K_{\cW}+(D_{\mathbb{A}^1})_\cW=\pi^*(K_{X_{\mathbb{A}^1}}+D_{\mathbb{A}^1})+A_{(X,D)}(S) \bar{C}.
\]
Denote $\hat{L}=\cO_{X}(V_\infty)$ for the cone construction, where $V_\infty$ is the divisor at infinity which is isomorphic to $V$. 
Then we have $-K_X-D=(1+r) \hat{L}$ and:
\[
K_{\cW}+(D_{\mathbb{A}^1})_\cW=-(1+r)\rho^*\hat{L}+A_{(X,D)}(S) \bar{C},
\]
where $\bar{C}$ is the orbifold cone over $C$ and $\rho\colon \cW\to X$ the composite of $\pi\colon \cW\to X_{\mathbb{A}^1}$ with the projection $X_{\mathbb{A}^1}\to X$. 

As in \cite{LL16}, we define the cone angle parameter $\beta=\frac{r}{n}$ and let $\delta=r\frac{n+1}{n}$. Then  
\begin{equation}\label{eq-alcX}
-(K_X+D+(1-\beta)V_\infty)\sim_{\bQ}(1+r)\hat{L}-(1-\frac{r}{n})\hat{L}=r\frac{n+1}{n}\hat{L}=\delta \hat{L}.
\end{equation}
Denote by  $\bV_{\infty}$ the birational transform of $(V_{\infty})_{\mathbb{A}^1}$ on $\cW$.  We also get:
\begin{eqnarray}\label{lclY2}
& &K_{\cW}+(D_{\mathbb{A}^1})_\cW+(1-\beta)\bV_\infty\nonumber\\
&=&\pi^*\big(K_{X_{\mathbb{A}^1}}+D_{\mathbb{A}^1}+(1-\beta) (V_\infty)_{\bA^1}\big)+A_{(X,D)}(S)\bar{C}\nonumber\\
&=&-\delta\rho^*\hat{L}+A_{(X,D)}(S)\bar{C}.
\end{eqnarray}

The above construction works for any Koll\'{a}r component. From now on we assume that $(V,\Delta)$ is K-semistable and $S$ minimizes the normalized volume, i.e. it satisfies
\begin{equation}
\hvol(\ord_{S})=\hvol(\ord_{V_0})=r^n (H^{{n-1}}),
\end{equation}
where $V_0$ denotes the exceptional divisor obtained by blowing up the vertex of the cone and we aim to show $S=V_0$. We note that by Theorem \ref{thm-Ksemi}, $\hvol(\ord_{V_0})$ is the minimal normalized volume. 
Then we have:
\[
\vol(\ord_{S})=\frac{\hvol(\ord_{S})}{A_{(X,D)}(S)^n}=\frac{r^n (H^{n-1})}{A_{(X,D)}(S)^n}.
\]

In Section \ref{s-equiv}, we have used the filtration induced by a valuation (see also \cite{BHJ15, Fuj15}).  Here we use the same construction but for sections on the projective cone instead of the base. 

\addtocounter{thm}{1}

\begin{defn}[Filtration by valuation]
For a fixed a valuation $v\in \Val_{X,o}$, let $\hat{R}_m=H^0(X, m \hat{L})$. Define $\cF^x\hat{R}_m:=\cF^x_v \hat{R}_m \subset \hat{R}_m$ to be a decreasing filtration (with respect to $x$) as follows:
$$\cF^x\hat{R}_m=H^0(X, m \hat{L} \otimes \fa_x), \qquad \mbox{where \ } \fa_x=\{f\in \cO_X\ |\  v(f)\ge x\}. $$

On $\bigoplus_{m=0} \hat{R}_m$, we define $\cF \hat{R}^{(t)} :=\cF_v\hat{R}^{(t)}= \bigoplus \cF^{kt}\hat{R}_k$.  Then the volume is defined to be
$$\vol(\cF \hat{R}^{(t)}) := \limsup_{m\to \infty} \frac{\dim_{\bC} (\cF^{mt}\hat{R}_m)}{m^n/n!}.$$

\end{defn}

The following proposition answers the question in \cite[Section 6]{LL16}. 
\begin{prop}\label{p-fanocone}
With the above notation, if the base $(V,\Delta)$ is log K-semistable, then $(X,D+(1-\beta)V_{\infty})$ is log K-semistable. As a consequence,
\[
A_{(X,D)}(S)-\frac{\delta }{(\hat{L}^{n})}\int_0^{+\infty}\vol(\cF_{\ord_S} \hat{R}^{(x)})dx\ge 0.
\]
\end{prop}

\begin{proof}

It is enough to verify that the generalized Futaki invariant is nonnegative for any compactified special test configuration $\pi\colon (\mathcal{X},\mathcal{D}+(1-\beta)\cV)\to \mathbb{P}^1$ of $(X,D+(1-\beta)V_{\infty})$ over $\mathbb{P}^1$ (see Section \ref{ss-tc}), where $\mathcal{V}\supset V_{\infty}\times( \mathbb{P}^1\setminus \{0 \})$ the closure. Let $\Delta_{\infty}(=\Delta)=V_{\infty}\cap D$ and $\Delta^{\rm tc}$ be the closure of $\Delta_{\infty}\times ( \mathbb{P}^1\setminus \{0\})$. Then  $\mu\colon(\mathcal{V},\Delta^{\rm tc})\to \mathbb{P}^1$ is a compactified test configuration of $(V,\Delta)$. 
As $(1+r)V_\infty\sim_{\mathbb{Q}}  -(K_X+D) $, we know that there exists $k\in \bQ$ such that:
\begin{eqnarray*}
&&(1+r)\mathcal{V}\sim_{\mathbb{Q}}-K_{\mathcal{X}}-\mathcal{D}+\pi^*\cO_{\bP^1}(k) \text{ and} \\
&&K_{\cV}+\Delta^{\rm tc}=(K_{\mathcal{X}}+\mathcal{D}+\cV)|_{\cV}=-r \cV|_{\cV}+\mu^*\cO_{\bP^1}(k).
\end{eqnarray*} 
The adjunction formula holds because $\mathcal{X}$  is smooth along the codimension 2 points over $0$ and so there is no different divisor. Since $\beta=\frac{r}{n}$ and $\delta=r\frac{1+n}{n}$, we have the identity:
$$
-(K_{\cX/\bP^1}+\cD+(1-\beta)\cV)\sim_{\bQ} \delta\cdot \cV+\pi^*\cO_{\bP^1}(-2-k).
$$
Then the generalized Futaki invariant of $(\mathcal{X},\mathcal{D}+(1-\beta)\cV)/\mathbb{P}^1$ is equal to:
\begin{eqnarray*}
{\rm Fut}(\mathcal{X})&=& -\frac{1}{(n+1)(\delta \hat{L})^n}(-K_{\mathcal{X}/\mathbb{P}^1}-\mathcal{D}-(1-\beta)\mathcal{V})^{n+1}\\
&=&-\frac{1}{\hat{L}^n }\pi^*\cO_{\bP^1}(-2-k) \cdot \cV^{n}-\frac{\delta}{(n+1)\hat{L}^{n}}\cV^{n+1}.
\end{eqnarray*}
On the other hand, the generalized Futaki invariant of $(\cV, \Delta^{\rm tc})/\bP^1$ is equal to:
\begin{eqnarray*}
{\rm Fut}(\cV)&=&-\frac{1}{nr^{n-1} H^{n-1}}((-K_{\mathcal{V}/\mathbb{P}^1}-\Delta^{\rm tc})|_{\mathcal{V}})^{n}\\
&=&-\frac{1}{nr^{n-1}H^{n-1}}(r\cV|_{\cV}-\mu^*\cO_{\bP^1}(k)+\mu^*K_{\bP^1})^n\\
&=&-\frac{1}{r^{n-1}H^{n-1}} r^{n-1}\mu^*\cO_{\bP^1}(-2-k) \cdot (\cV|_{\cV})^{n-1}-\frac{r}{n H^{n-1}}(\cV|_\cV)^{n}\\
&=&-\frac{1}{H^{n-1}}\pi^*\cO_{\bP^1}(-2-k)\cdot \cV^n-\frac{r}{n H^{n-1}}\cV^{n+1}.
\end{eqnarray*}
Because $H^{n-1}=\int_{ [V]} H^{n-1}=\int_{[X]} \hat{L}^n=\hat{L}^n$, we have the identity: 
$${\rm Fut}(\mathcal{V})={\rm Fut}(\mathcal{X}). $$
Finally, recall the log-K-semistability is equivalent to the log-Ding-semistablity (see e.g. \cite{Fuj16}). Then the second statement is obtained by applying \cite[Proposition 4.5]{LL16} to $(X,D+(1-\beta)V_{\infty})$ and $\hat{L}=-\frac{1}{\delta}(K_X+D+(1-\beta)V_{\infty})$. 
\end{proof}

The following calculations are key to us and proved in \cite[Proof of Proposition 4.5]{LL16}.
\begin{prop}[\cite{LL16}]\label{p-equal}
Suppose $(V, \Delta)$ is log-K-semistable. If $S$ is a Koll\'{a}r component obtaining the minimum of $\hvol$ over $(X, o)$, then the graded filtration induced by $S$ satisfies the following two conditions:
\begin{enumerate}
\item The following identity holds:
\[
A_{(X,D)}(S)-\frac{\delta }{(\hat{L}^{n})}\int_0^{+\infty}\vol(\cF \hat{R}^{(x)})dx=0.
\]
\item Denote $\tau:=\sqrt[n]{\frac{(\hat{L}^{ n})}{\vol(\ord_{S})}}$. We have:
\[
\vol\left(\cF \hat{R}^{(x)}\right)=\vol_Y(\mu^*\hat{L}-xS)=(\hat{L}^{n})-\vol(\ord_{S})x^n
\text{ for any } x\in [0, \tau].
\]
\end{enumerate}
\end{prop}

\begin{lem}We have $\tau=\frac{A_{(X,D)}(S)}{r}$.
\end{lem}
\begin{proof}
Combining 1 and 2 in Proposition \ref{p-equal}, we know that 
\[
A_{(X,D)}(S)-\frac{r(1+n)}{n\cdot \hat{L}^n}\int^{\tau}_{0}\big(\hat{L}^n-\vol(\ord_S)x^n\big)dx= A_{(X,D)}(S)-r\cdot \tau= 0.
\]
\end{proof}

Arguing as in \cite{Fuj15} (see also \cite{Liu16}), we know that:
\begin{lem}\label{lemamp}
$\tau$ is the nef threshold of $\mu^*\hat{L}$ with respect to the divisor $S$, i.e. 
$$\tau=\sup\left\{x\;|\; \mu^*\hat{L}-xS \text{ is ample } \right\}.$$ 
\end{lem}
\begin{proof} When the point is smooth, this follows from \cite[Theorem 2.3(2)]{Fuj15}. Exactly the same argument can be used to treat the current case. 
\end{proof}
\begin{thm}\label{t-cone}
If $S$ is a Koll\'{a}r component that obtains the minimum of the normalized volume, then $S$ is the canonical component $V_0$.
\end{thm}

We first show the following statements.
\begin{lem}
\begin{enumerate}
\item
$\rho^*\hat{L}-\tau \bar{C}$ is semi-ample, and contracts $Y$ to $S_{\infty}(\cong S)\subset \bar{C}$ as the divisor at infinity of the orbifold projective cone $\bar{C}=\bar{C}(S, -S|_S)$.
\item
$A_{(X,D)}(S)=r$ and 
 there is a special test configuration $\cX$ of $(X, D+ (1-\beta)V_\infty; \hat{L})$ whose central fibre $X_0$ is 
$(\bar{C}, {C}_D+(1-\beta) S_{\infty}; \hat{L}_0)$ where ${C}_D$ is the intersection of $ \bar{C}$ with $(D\times \bA^1)_\cW$. Moreover, $(\bar{C}, {C}_D+(1-\beta) S_{\infty}; \hat{L}_0)\cong (X, D+ (1-\beta)V_\infty; \hat{L})$.
\end{enumerate}
\end{lem}
\begin{proof}

The proof of this part is along the similar line in \cite[Proof of Lemma 33]{Liu16}.
First we observe the following restrictions of $\rho^*\hat{L}-x\bar{C}$:
\begin{enumerate}
\item
$\left.\rho^*\hat{L}-x\bar{C}\right|_{X_t}=\hat{L}$, $t\neq 0$. Recall that $X_t\cong X$ for $t\in \bC^*$.
\item
$\left.\rho^*\hat{L}-x \bar{C}\right|_{Y_0}=\mu^*\hat{L}-x S$. 
\item
$\left.\rho^*\hat{L}-x\bar{C}\right|_{\bar{C}}=-x\bar{C}|_{\bar{C}}=x Y_0|_{\bar{C}}=x S_{\infty}=x\cO_{\bar{C}}(1)$. 
\end{enumerate}
So by Lemma \ref{lemamp}, it is easy to see that $\rho^*\hat{L}-x\bar{C}$ is ample when $x\in (0, \tau)$. 
To show that $\rho^*\hat{L}-\tau\bar{C}$ is semi-ample, we use \eqref{lclY2} to calculate:
\begin{eqnarray*}
m(\rho^*\hat{L}-x \bar{C})-K_{\cW}-(D_{\bA^1})_\cW&=&m(\rho^*\hat{L}-x\bar{C})+(1+r) \rho^*\hat{L}-A_{(X,D)}(S)\bar{C}\\
&=&(m+1+r)\left(\rho^*\hat{L}-\frac{mx+A_{(X,D)}(S)}{m+1+r}\bar{C}\right).
\end{eqnarray*}
Notice that:
\[
\frac{mx+A_{(X,D)}(S)}{m+1+r}<\tau= \frac{A_{(X,D)}(S)}{r}
\]
if and only if 
\[
x< \left(1+\frac{1}{m}\right)\frac{A_{(X,D)}(S)}{r}.
\] 
Because this is satisfied for 
$$x=\tau=\frac{A_{(X,D)}(S)}{r}\qquad \mbox{for any\ } m>0,$$ the first statement holds by base-point-free theorem \cite[Theorem 3.13]{KM98}. 
Next we claim that 
\begin{equation}\label{eqcontr1}
H^0(Y, m(\mu^*\hat{L}-\tau S))\cong H^0(S, -m\tau S)
\end{equation} 
for any $m$ sufficiently divisible. To see this, we consider the exact sequence:
\begin{equation}
0\rightarrow \cO_{Y}(m(\mu^*\hat{L}-\tau S)-S)\rightarrow \cO_{Y}(m(\mu^*\hat{L}-\tau S))\rightarrow \cO_{Y}(m(\mu^*\hat{L}-\tau S))\otimes\cO_S\rightarrow 0,
\end{equation}
and its associated long exact sequence of cohomology groups.
By the above discussion, and
\[
m(\mu^*\hat{L}-\tau S)-S-K_Y=m(\mu^*\hat{L}-\frac{A_{(X,D)}(S)}{r}S)+(1+r)\mu^*\hat{L}-A_{(X,D)}(S)S
\]
is ample,  it follows from the Kawamata-Viehweg vanishing theorem that  
\[
H^1\big(Y, m(\mu^*\hat{L}-\tau S)\otimes \cO(-S)\big)=0 \text{ for any } m\ge 0.
\]
We also have
\[
H^0\big(Y, m(\mu^*\hat{L}-\tau S)\otimes \cO(-S)\big)=0 \text{ for any } m\ge 0,
\]
as $\tau$ is also the pseudo-effective threshold. 
Thus we know $|m(\rho^*\hat{L}-\tau \bar{C})|$ contracts the fiber  $\cW\times_{\mathbb{A}^1} \{0 \}$ to $\bar{C}$ for sufficiently divisible $m$. This finishes the proof of (1). We denote by $ \theta\colon \cW\to \mathcal{X}$ the induced morphism and there is an ample line bundle $\hat{\cL}$ on $\cX$ such that $\theta^*\hat{\cL}=\rho^*\hat{L}-\tau \bar{C}$. 
\bigskip

Next we prove (2).  Let $(D_{\mathbb{A}^1})_{\cX}$ be the push forward of $(D_{\mathbb{A}^1})_{\cW}$ on $\cX$. Then $-K_{\cX}-(D_{\mathbb{A}^1})_{\cX}$ and $(1+r) \hat{\cL}$ coincide outside $X_0$, they must be relatively linearly equivalent on the whole
$\cX$ because $X_0$ is irreducible. In particular, they are linearly equivalent when restricted to $X_0$.

Since 
$$ (K_{Y}+\mu^{-1}_*D+S)|_S=K_S+\Delta_S\sim_{\mathbb{Q}}A_{(X,D)}(S) \cdot S|_S,$$ 
we know that
$$-K_{\cX}-(D_{\mathbb{A}^1})_{\cX}|_{X_0}=-K_{\bar{C}}-{C}_D\sim_{\mathbb{Q}}(1+A_{(X,D)}(S))S_{\infty}.$$
Similarly, we have
$\hat{\cL}|_{X_0}\sim_{\bQ}\tau S$ with $\tau=\frac{A_{(X,D)}(S)}{r}$.
Therefore,
\[
1+A_{(X,D)}(S)=(1+r)\frac{A_{(X,D)}(S)}{r},
\]
which implies $A_{(X,D)}(S)=r$ and $\tau=1$. 

The degree of $V_{\infty}$ under $\hat{\cL}$ is
\begin{eqnarray*}
\hat{\cL}|_{X_0}^{n-1}\cdot V_\infty&=&\hat{L}^{n-1}\cdot V_\infty\\
&=&\hat{L}^{ n},
\end{eqnarray*}
while the degree of $S$ is
\begin{eqnarray*}
\hat{\cL}|_{X_0}^{ n-1}\cdot S &=&\tau^{-1}\hat{\cL}|_{X_0}^{ n}=\hat{L}^n=\hat{L}_0^n.
\end{eqnarray*}

The restriction $\theta|_{V_{\infty}} \colon V_{\infty}\to S $ is finite since 
$$(\rho^*\hat{L}-\tau \bar{C})|_{V_{\infty}}=\hat{L}|_{V_{\infty}}$$
is ample. And the degree is one by the above calculation on degrees, which implies this is an isomorphism. We claim that $Y$ is indeed the $\mathbb{P}^1$-bundle over $V_{\infty}$ induced by blowing up the vertex of $X$,  $S$ is a section and the morphism $\theta$ is just contracting the $\mathbb{P}^1$-bundle. 
 Granted this for now, 
we then indeed have an isomorphism from
$(X, D+(1-\beta)V_\infty; \hat{L})$
to $(\bar{C}, C_D+(1-\beta) S_{\infty}; \hat{L}_0)$. 

To see the claim, let $l$ be a curve contracted by $\theta$, we want to show that it is the birational transform of a ruling line of $X$. To see this, since $(\rho^*(\hat{L})-\bar{C})\cdot l=0$, we know that $\rho^*(\hat{L})\cdot l=1$. So the image $\rho_*l$ of $l$ in $X$ is a line, and it passes through the vertex. Therefore, it is a ruling of the cone. 

\end{proof}

By the above proof, let $\cV$ be the birational transform of $(V_{\infty})_{\mathbb{A}^1}$ on $\cX$, and $\cH$ the extension of $H_{\mathbb{A}^1}\setminus\{0\}$ on $\cX$, we know that:
\[
\cX={\rm Proj}_{\cV}\left( \bigoplus_{k=0}^{+\infty} \cS_k\right),
\]
where
\[
\cS_k=\bigoplus_{i=0}^{k} (H^0(\cV, i \cH)\cdot u^{k-i}).
\]
From this we easily see that 
$S$ and $V$ give the same component over the vertex. 


\subsection{The general case}

In this section, we prove Theorem \ref{t-main2} in the general case. 
We first show that the cone case we prove in Section \ref{ss-ucone} can be generalized to an orbifold cone.  Let $T=\mathbb{C}^*$.

\begin{prop}\label{p-Tequiv}
Let $o\in (X,D)$ be a klt $T$-singularity. Assume a minimizer $v$ of $\hvol_{(X,D),o}$ is given by a rescaling of $\ord_{S}$ for a Koll\'ar component $S$. Then $v$ is $T$-invariant. 
\end{prop}
\begin{proof}
 Let $\fa$ be an ideal whose normalized blow up gives the model of extracting the Koll\'ar component $S$ (see the proof of Theorem \ref{p-equality}). Denote the degeneration of $\fa_{\bullet}:=\{\fa^p\}$ induced by the $T$-action by $\fb_\bullet:=\{\bin (\fa^p)\}$ (which in general is not necessarily equal to but only contains $(\bin(\fa))^p$ ) by a sequence $\{\kB_{\bullet}\}$ of flat families of ideals over $\bA^1$.  
 
 Because $\ord_S$ is a minimizer of $\hvol=\hvol_{(X,D),o}$, we have:
 $$\mult(\fa)\cdot \lct^n(X,D; \fa)=\hvol(\ord_S)\le \mult(\fb_{\bullet})\cdot \lct^n(X,D; \fb_{\bullet}).$$
  But $\mult(\fa)=\mult(\fb_{\bullet})$, and $ \lct(X,D,\fa)\ge  \lct(X,D,\fb_{\bullet})$ by the lower semicontinuity of log canonical thresholds. So we know that 
  $$ \lct(X,D,\fa)= \lct(X,D,\fb_{\bullet})=\lim_{k\to \infty}\lct(X,D, \frac{1}{k}\fb_k)=:c.$$
Because $S$ computes the log canonical threshold of $\fa$, we have $A(S; X,D+c\cdot \fa)=0$.
As a consequence, we can choose $\epsilon$ sufficiently small, and $k$ sufficiently large, such that the log discrepancy
$$A(S; X,D+\frac{c-\epsilon}{k}\fa^k)<1\qquad \mbox{and \ \ $({X},D+\frac{c-\epsilon}{k}\fb_k)$ is klt. }$$
This implies that  $(X_{\mathbb{A}^1}, D_{ \mathbb{A}^1}+\frac{c-\epsilon}{k}\kB_k)$ is klt and $A( S_\cY; X_{\mathbb{A}^1}, D_{ \mathbb{A}^1}+\frac{c-\epsilon}{k}\kB_k)< 1$, where  $S_\cY$ is the divisor birational to $S\times \mathbb{A}^1$. Thus  by \cite{BCHM10} we can construct a model $\mu_{\cY}: \cY\to X_{\bA^1}$ extracting only the irreducible divisor $ S_\cY$ which gives $S$ over the generic fiber. Furthermore, we can assume $-S_\cY$ is ample over $X_{\bA^1}$.

Denote $Y=\cY\times_{\mathbb{A}^1}\{1\}$ and by $\mu_{Y}\times {\rm id}: Y\times\bA^1\rightarrow X_{\bA^1}$ the extraction of $S\times\bA^1$. Then $\cY$ and  $Y\times \mathbb{A}^1$  are isomorphic in codimension 1, with anti-ample exceptional divisors over $X_{\bA^1}$. Thus we conclude that they are isomorphic:
\begin{eqnarray*}
\cY&=&{\rm Proj}_{X_{\bA^1}}\bigoplus_m(\mu_\cY)_* \mathcal{O}_{\cY}\big(-m S_{\cY}\big)\\
& =& {\rm Proj}_{X_{\bA^1}}\bigoplus_m (\mu_{Y}\times{\rm id})_* \mathcal{O}_{Y\times \bA^1}\big(-m(S\times \bA^1)\big)\\
&=&Y\times \bA^1.
\end{eqnarray*} 
\end{proof}

\begin{prop}\label{p-cover2}
Under the notation  in Section \ref{ss-deformation}, $S$ is the unique minimizer among all Koll\'ar components for $\hvol_{({C},C_D)}$ if and only if the same holds for ${C}^{(d)}$ on
$\hvol_{({C}^{(d)},C^{(d)}_{1}+C^{(d)}_2)}$.
\end{prop}
\begin{proof}By Proposition \ref{p-Tequiv}, any minimizing Koll\'ar component  $E$ of $\hvol_{(C,C_D)}$ is $T$-invariant. Therefore it is $G=\mathbb{Z}/d$ equivalent. So $E$ is the pull back of a Koll\'ar component on $C^{(d)}$ by Lemma \ref{l-finite}, which can only be the canonical component obtained by blowing up the vertex by our assumption and Lemma \ref{l-finitevolume}. 
\end{proof}

\begin{proof}[Proof of Theorem \ref{t-main2}] 
By Theorem \ref{t-cone} and Proposition \ref{p-cover2}, we know that for the coarse moduli space of an orbifold cone over a K-semistable log Fano pair, the only Koll\'ar component which minimizes the normalized volume function is given by the canonical component. 

Now we consider the case of a general klt singularity $o\in (X,D)$. Because $(S, \Delta_S)$ is K-semistable, by Theorem \ref{t-main}, $\ord_S$ obtains the minimum of $\hvol_{(X,D),o}$. Let us assume that there is another divisor $F$ over $(X,o)$ such that 
$$\hvol_{(X,D),o}(\ord_F)=\hvol_{(X,D),o}(\ord_S).$$
Then by Theorem \ref{t-divisor}, $F$ is indeed a Koll\'{a}r component. 

As in Section \ref{ss-deformation}, let $\pi\colon \cW\to X\times \mathbb{A}^1$ be the flat family which degenerates $X$ to $W_0=Y_0\cup \bar{C}$, where $Y_0\cong Y$ extracts $S$ over $X$ and $\bar{C}$ is the the coarse moduli space of the orbifold cone  over $S=\bar{C}\cap Y_0$. Then as in the proof Proposition \ref{p-Tequiv}, let $\fa$ be an ideal whose normalized blow up gives the model of extracting the Koll\'ar component $F$ (see the proof of Theorem \ref{p-equality}). Denote the degeneration of $\{\fa_{\bullet}\}:=\{\fa^p\}$ by $\fb_{\bullet}:=\{\bin (\fa^p)\}$.  
 
Denote by $S_0$ the induced Koll\'{a}r component over $(C, C_D, o)$ (see Proposition \ref{p-degvol}). We then have:
 $$\mult(\fa)\cdot \lct^n(X,D, \fa)=\hvol_{X}(\ord_F)= \hvol_{\bar{C}}(\ord_{S_0})\le \mult(\fb_{\bullet})\cdot \lct^n(\bar{C},\bar{C}_D, \fb_{\bullet}),$$
 where the last inequality is from the assumption that $S_0\cong S$ is K-semistable and Theorem \ref{t-main}. 
 On the other hand, we have $\mult(\fa)=\mult(\fb_{\bullet})$ and $\lct(X,D,\fa)\ge  \lct(\bar{C},\bar{C}_D, \fb_{\bullet})$ (see the proof of Lemma \ref{l-deg}). So we know that 
  $$ \lct(X,D,\fa)= \lct(\bar{C},\bar{C}_D, \fb_{\bullet})=\lim_{k\to \infty}\lct(\bar{C}, \bar{C}_D, \frac{1}{k}\fb_k),$$
which we denote by  $c$. 
In particular, for we can choose $\epsilon$ sufficiently small, and $k$ sufficiently large, such that the log discrepancy
$$A(F; X,D+(c-\epsilon)\fa)<\delta \qquad \mbox{for sufficiently small $\delta>0$}$$
and $({C},C_D+(c-\epsilon)\frac{1}{k}\fb_k)$ is klt. 
Thus  similar to the proof of Proposition \ref{p-Tequiv}, by \cite{BCHM10} we can construct a model $\psi'_1: \cZ'\to \cW$ extracting only the irreducible divisor $F_{\cZ'}$ which gives $F$ over the generic fiber. Furthermore, we can assume $-F_{\cZ'}$ is ample over $\cW$. 
 
We claim that the special fiber $Z'_0\to W_0$ is a normal model which also only extracts a Koll\'ar component over $\bar{C}$. In fact, let $\nu: (Z'_0)^{\rm{n}}\to Z'_0$ be the normalization and $\rho\colon (Z'_0)^{\rm{n}}\to W_0$ be the composite morphism.  Locally over the vertex $v$ of $\bar{C}$,  we have
\begin{eqnarray}\label{e-adjun}
\nu^*((K_{\cZ'}+Z'_0+F_{\cZ'}+(\phi'^{-1})_{*}D_{\bA^1})|_{Z'_0})&=:&K_{(Z_0')^{\rm n}}+G+\rho_*^{-1}C_D\nonumber\\
&\ge& K_{(Z_0')^{\rm n}}+{\rm Ex}(\rho)+\rho_*^{-1}C_D
\end{eqnarray}
by \cite[Proposition 4.5]{Kol13}.
Denote the pull back of $F_0$ on $(Z'_0)^{\rm{n}}$ by $\tilde{F}_0$. Then we have:
\begin{eqnarray*}
\hvol_{({C}, {C}_D),o_C}(\ord_{S_0})&=&\hvol_{(X,D),o}(\ord_F)\\
&=&(-(K_{\cZ'}+F_{\cZ'}+(\phi'^{-1})_{*}D_{\bA^1})|_{F})^{n-1}\\
&=&(-\nu^*((K_{\cZ'}+F_{\cZ'}+(\phi'^{-1})_{*}D_{\bA^1})|_{Z'_0})|_{\tilde{F}_0})^{n-1}\\
&\ge &(-\nu^*((K_{\cZ'}+F_{\cZ'}+(\phi'^{-1})_{*}D_{\bA^1})|_{Z'_0})|_{(\tilde{F}_0)_{\rm red}={\rm Ex}(\rho)})^{n-1}\\
&\ge &\vol_{(C,C_D),o_C}((Z'_0)^{\rm n}) \ \ \ \mbox{(by Definition \ref{d-modelv} and \eqref{e-adjun})}.
\end{eqnarray*}
Thus we conclude that the volume of the model $(Z'_0)^{\rm n}$ is equal to the minimum of the normalized volume $\hvol_{{(C,C_D)},o_C}$.  It follows from the argument in the proof of Theorem \ref{t-divisor} that $\tilde{F}_0$ is reduced and yields a Koll\'ar component over $o_C$. This implies that $\nu$ is isomorphic along the generic point of $\tilde{F}_0$, and thus $Z'_0$ is regular along the generic point of $F_0$. Since $Z'_0$ is Cohen-Macaulay, we conclude that $Z'_0$ is normal by Serre's criterion.   By the proof in the cone case Theorem \ref{t-cone}, $F_0$ has to be the same as the canonical component $S_0$. The rest of the construction can be seen as the reverse of the construction at the beginning of section \ref{ss-ucone}. 

In particular, the birational transform $\bP'$ of $\bar{C}$ in $\cZ'$ is the extraction of the canonical component.  
Thus there is a morphism $\mathbb{P}'\to S$. Let  $l$ be the fiber class of $\mathbb{P}'\to S$. Consider $K_{\cZ'}+\phi_*^{-1}(D_{\mathbb{A}^1})+F_{\cZ'}$, which satisfies that 
$$\big(K_{\cZ'}+\phi_*^{-1}(D_{\mathbb{A}^1})+F_{\cZ'}\big)|_{F_0}=K_{F_0}+\Delta_{F_0}$$
is anti-ample, 
$$ \big(K_{\cZ'}+\phi_*^{-1}(D_{\mathbb{A}^1})+F_{\cZ'}\big)|_S=K_{S}+\Delta_S$$
is anti-ample, and  
$$\big(K_{\cZ'}+\phi_*^{-1}(D_{\mathbb{A}^1})+F_{\cZ'}\big)\cdot l=0.$$
Thus $l$ is an extremal ray in $N_1(\cZ'/X_{\mathbb{A}^1})$. 

Hence we know that there is a morphism $\phi'_1: \cZ'\to \cY'$ which contracts $\mathbb{P}'$, and $\cY'$ admits a morphism $\chi\colon \cY'\to X_{\mathbb{A}^1}$. Restricting over $0$, the central fiber is the birational model $\mu\colon Y\to X$ which extracts $S$. On the other hand, if we let $\mu_{Y_F}\colon Y_F\to X$ be the birational model which extracts the Koll\'ar component $F$, then $\mu_{Y_F}\times {\rm id}: Y_F\times\bA^1\rightarrow X\times{\bA^1}$ is isomorphic to the restriction of $\chi$ over $t\neq 0$. Note that by construction we have the following commutative diagram (in which the first column is not used in the current proof but only in the previous arguments):
\begin{equation*}
\xymatrix @R=2pc @C=2pc
{
\cZ \ar^{\psi_1}[r] \ar^{\phi_1}[d] \ar^{\phi} [rd]   & \cW  \ar^{\pi}[d] &  \ar_{\psi'_1}[l] \cZ' \ar^{\phi'_1}[d]  \ar_{\phi'} [ld]&\\
Y\times\bA^1 \ar^{\mu_{\bA^1}}[r] & X\times\bA^1  & \ar_{\chi}[l] \cY' \ar@{-->}[r] & Y_F\times \bA^1 \\
Y\times\{0\} \ar[r] \ar@{^{(}->}[u] & X\times\{0\} \ar@{^{(}->}[u] &\ar[l]  \ar@{^{(}->}[u] Y\times\{0\} &
}
\end{equation*}

As $Y_F\times \mathbb{A}^1$ and $\cY'$ is isomorphic in codimension 1, if we denote by $F_{\cY'}$ the push forward of $F_{\cZ'}$ on $\cY'$, we have
\begin{eqnarray*}
\cY' &=&{\rm Proj}_{X_{\bA^1}}\bigoplus_m{\chi}_* \mathcal{O}_{\cY'}\big(-m F_{\cY'}\big)\\
& =& {\rm Proj}_{X_{\bA^1}}\bigoplus_m{(\mu_{Y_F}\times{\rm id})}_* \mathcal{O}_{Y_F\times \bA^1}\big(-m(F\times \bA^1)\big)\\
&=&Y_F\times \bA^1.
\end{eqnarray*}
Considering the central fiber over $0$, this implies that $Y_F=Y$ and hence $F=S$. 
\end{proof}

\section{Minimizing Koll\'{a}r component is K-semistable}\label{s-Ksta}

In this section, we aim to prove the a Koll\'ar component is minimizing only if it is K-semistable. The method used in the proof of this result is motivated by the one in the study of toric degenerations (see e.g. \cite[Section 3.2]{Cal02}, \cite[Proposition 2.2]{AB04} and \cite[Proposition 3]{And13}). In particular the argument allows us to reduce two-step degenerations to a one-step degeneration.

\begin{proof}[Proof of Theorem \ref{t-mintok}]
Let $(X, D, o)$ be a klt singularity with $X={\rm Spec} (R)$. Assume that $S$ is a Koll\'{a}r component that minimizes $\hvol_{(X, D),o}$ and appears as the exceptional divisor in a plt blow-up $\mu: Y\rightarrow X$. Let $\Delta_S$ be the divisor on $S$ satisfying $K_{Y}+(\mu^{-1})_*D+S|_S=K_S+\Delta_S$. 
By Theorem  \ref{thm-Ksemi} and \ref{p-cover}, to show that $(S,\Delta_S)$ is K-semistable, it suffices to show that the canonical component is a minimizer of $\hvol_{(C,C_D)}$, where $(C, C_D)$ is the degeneration  associated to $S$ (see the degeneration construction in Section \ref{ss-deformation} and \ref{ss-ucone}).

By Proposition \ref{l-Tmini}, we only need to show that  
$$\hvol_{(C,C_D)}(\ord_{S_0}) \le \hvol_{(C,C_D)}(\ord_{F})$$
 for any $\bC^*$-invariant Koll\'ar component $F$ over the vertex $o_C\in (C,C_D)$. 
 Let $(\cC, \cE)$ be the associated special degeneration which degenerates $(C,C_D)$ to a pair $(C_0, E_0)$ where $C_0$ is an orbifold cone over $(F, \Delta_F)$ (see Section \ref{ss-deformation}).
 Then we have a $ \mathbb{Z}_{\ge 0}\times \mathbb{Z}_{\ge 0}$-valued order function, which yields a rank-2 valuation, defined on $R$:
\begin{eqnarray}\label{eq-Z2val}
w: R& \longrightarrow &  \mathbb{Z}_{\ge 0} \times \mathbb{Z}_{\ge 0}\\
f & \mapsto &\left( \ord_S(f), \ord_F(\bin(f)) \right).\nonumber
\end{eqnarray}
We give $  \mathbb{Z}_{\ge 0}\times \mathbb{Z}_{\ge 0}$ the following lexicographic order: $(m_1, u_1)<(m_2, u_2)$ if and only if $m_1<m_2$, or $m_1=m_2$ and $u_1<u_2$. If we denote by $\Gamma$ the valuative monoid and denote the associated graded ring by
\[
\gr_w R=\bigoplus_{(m, u)\in \Gamma} R_{\ge (m, u)}/ R_{> (m, u)},
\]
then it is easy to see that $C_0={\rm Spec}_{\bC} \left(\gr_w R\right)$. 
We will also denote:
\[
R^*=\bigoplus_{m\in \mathbb{N}} R_{\ge m}/R_{>m}=\bigoplus_{m\in \mathbb{N}} R^*_m.
\]
Then ${\rm Spec}(R^*)=C$ and $\gr_w R=\gr_{\ord_F}(R^*)$. Moreover if we define the extended Rees ring of $R^*$ with respect to the filtration associated to $\ord_{F}$ 
(see Section \ref{s-degin}):
\[
\mathcal{A}':=\bigoplus_{k\in \bZ} \cA_k:=\bigoplus_{k\in \bZ} \mathfrak{b}_k t^{-k} \subset R^*[t, t^{-1}],
\]
where $\mathfrak{b}_k=\{f\in R^*; \ord_{F}(f)\ge k\}$. Then the flat family $\cC \rightarrow \bA^1$ is given by ${\rm Spec}_{\bC[t]}\left(\mathcal{A}'\right)$. In particular, we have
\[
\cA' \otimes_{\bC[t]}\bC[t, t^{-1}]\cong R^*[t, t^{-1}], \quad \cA' \otimes_{\bC[t]}\bC[t]/(t)\cong \gr_w R=\gr_{\ord_F}(R^*).
\]

\bigskip

Pick up a set of homogeneous generators $\bar{f}_1, \dots, \bar{f}_p$ for $\gr_w R$ with $\deg(\bar{f}_i)=(m_i, u_i)$. Lift them to generators $f_1, \dots, f_p$ for $R^*$ such that $f_i\in R^*_{m_i}$. 
Set $P=\mathbb{C}[x_1, \dots, x_p]$ and give $P$ the grading by $\deg(x_i)=(m_i, u_i)$ so that the surjective map 
$$\rho_0\colon P\rightarrow \gr_w R \qquad\mbox{ given by }\qquad x_i\mapsto f_i$$ is a map of graded rings. 
Let $\bar{g}_1, \dots, \bar{g}_q\in P$ be a set of homogeneous generators of the kernel ${\rm Ker}(\rho_0)$ and assume $\deg(\bar{g}_j)=(n_j, v_j)$. 

Since $\bar{g}_j(\bar{f}_1, \dots, \bar{f}_p)=0 \in {\rm gr}_wR$, it follows 
$$\bar{g}_j(f_1, \dots, f_p) \in (R_{n_j})_{>v_j} \qquad \mbox{ for each } j.$$ 
By
the flatness of $\cA'$ over $\bC[t]$, there exist liftings $g_j\in \bar{g}_j+(P_{n_j})_{> v_j}$ of the relation $\bar{g}_j$ such that 
$$g_j(f_1, \dots, f_p)=0  \mbox{ for } 1\le j\le q .$$ 
So $g_j{}'s$ form a Gr\"{o}bner basis of $J$ with respect to the order function $\ord_F$, where $J$ is the kernel of the 
surjection $P\rightarrow R^*$. In other words, if we let $K=(\bar{g}_1, \dots, \bar{g}_q)$ denote the kernel $P\rightarrow {\rm gr}_wR$, then $K$ is the initial ideal of $J$ with respect to the order determined by 
$\ord_F$. As a consequence, we have:
\[
\cA'=P[\tau]/(\tilde{g}_1, \dots, \tilde{g}_q),
\]
where $\tilde{g}_j=\tau^{v_j} g_j(\tau^{-u_1} x_1, \dots, \tau^{-u_p} x_p)$.

\bigskip

Now we lift $f_1, \dots, f_p$ further to generators $F_1, \dots, F_p$ of $R$. Then we have: 
$$g_j(F_1, \dots, F_p)\in R_{> n_j}.$$ 
Let $\cR'$ be the extended Rees algebraic associated to $ \ord_S$ on $R$ (see Section \ref{s-degin}). By the flatness of $\cR'$ over $\bC[t]$, there exist $G_j\in g_j+P_{> n_j}$ such that 
$$G_j(F_1, \dots, F_p)=0.$$ Let $I$ be the kernel of $P\rightarrow R$. Then $G_j{}'s$ form a Gr\"{o}bner basis with respect to the valuation $\ord_S$ and the associated initial ideal is $J$. As a consequence, we have:
\[
\cR'=P[\zeta]/(\tilde{G}_1, \dots, \tilde{G}_q)
\]
where $\tilde{G}_j=\zeta^{n_j} G_j(\zeta^{-m_1} x_1, \dots, \zeta^{-m_p} x_p)$. 

In summary, we have a $(\bC^*)^2$-action on $\bC^p$ generated by two 1-parameter subgroups $\lambda_0(t)=t^{\bold m}$ and $\lambda'(t)=t^{\bold u}$ where ${\bold m},{\bold u}\in \mathbb{N}^{p}$. $\lambda_0$ degenerates $(X, D)$ to $(C, C_D)$ and $\lambda'$ degenerates $(C, C_D)$ further to $(C_0, E_0)$. 

\bigskip

\begin{lem} For $0<\epsilon\ll 1$ and $\epsilon\in \bQ$, there is a family (parametrized by $\epsilon$) of subgroups $\lambda_{\epsilon}\colon \mathbb{C}^*\to (\mathbb{C}^*)^{2}$ such that $\lambda_\epsilon(t)$ degenerates $X$ to $C_0$ as $t\to 0$.
\end{lem}

\begin{proof}
Let $(n'_j, v'_j)$ be a degree of any homogeneous component of $G_j-\bar{g}_j$ and consider the difference $(n'_j, v'_j)-(n_j, v_j)$. Note that $(n'_j, v'_j)>(n_j, v_j)$. 
Denote by $B\subset \mathbb{Z}\times \mathbb{Z}$ the finite set consisting of such differences $(n'_j, v'_j)-(n_j, v_j)$, together with $0$ and the two generators of $\mathbb{N}\times\mathbb{N}$. Let $M$ be a positive integer that is larger than all absolute values of coordinates of pairs $B$ and let $\epsilon$ be sufficiently small such that $1> M \epsilon$. 
After tensoring with $\mathbb{Q}$, we can define
$$\pi_\epsilon=e^*_0+\epsilon e^*_{1}\colon \mathbb{Q}^2\to \mathbb{Q},$$
where $e^*_0$ and $e^*_1$ denote the first and second projection on the product $\mathbb{Z}^2=\bZ\times\bZ$. 
 
For $\epsilon> 0$ suficiently small, we define $\lambda_\epsilon: \bC^*\rightarrow {\rm GL}(p,\bC)$ to be the one parameter subgroup corresponding to the prime integral vector $N_{\epsilon}\cdot\pi_{\epsilon}$ in $\mathbb{Q}_{>0}\cdot \pi_\epsilon$: 
$$
\lambda_\epsilon(t)\cdot z=(t^{N_\epsilon(m_1+\epsilon u_1) } z_1, \cdots, t^{N_\epsilon (m_p+\epsilon u_p)}z_p) \text{ for any } (z_1,\dots, z_p)\in \bC^p.
$$
Note that in this setting, $\pi_0$ corresponds to $\ord_S$. 

Now to see that $\lambda_\epsilon$ degenerates $X$ to $C_0$, note that for any monomial $x^{\bold p}$ with bidegree $(m,u)=(\bold{ p\cdot m, p\cdot u})$, its degree under $\lambda_{\epsilon}$ is given by ${N}_{\epsilon}\cdot\pi_{\epsilon}(m,u)$. Then from our construction, we have $$\pi_\epsilon (n'_j, v'_j)> \pi_\epsilon (n_j, v_j)$$ 
where $(n'_j, v'_j)$ is a degree of any homogeneous part of $G_j-\bar{g}_j$ such that $(n'_j, v'_j)>(n_j, v_j)$.
Thus it follows that the initial term of $G_j$ with respect to the weight function $\pi_\epsilon$ is exactly $\bar{g}_j$. 
\end{proof}

Fix any $\lambda_{\epsilon}:\mathbb{C}^*\to (\mathbb{C}^*)^2$ for $0<\epsilon\ll 1$ as above. Then $\bC^*$ acts on $C_0$ via $\lambda_{\epsilon}$ and $C_0\setminus \{o_{C_0}\}$ is a $\bC^*$-Seifert bundle (see \cite{Kol04}) where $o_{C_0}$ is the vertex of $C_0$. We claim that the quotient $(C_0\setminus \{o_{C_0}\})/\lambda_\epsilon$ (which we will simply denote by $C_0/\lambda_{\epsilon}$) yields a Koll\'{a}r component $S_\epsilon$ over $(C_0,E_0)$. Furthermore, it induces Koll\'ar components over $(C, C_D)$ and $(X, D)$ which are both isomorphic to $S_\epsilon$ and such that the associated degenerations degenerate $(C, C_D)$ and $(X,D)$ to $(C_0, E_0)$.  
By abuse of notations we will also denote these Koll\'{a}r components over $(C, C_D)$ and $(X, D)$ by $S_{\epsilon}$.

Assuming this claim is true for now, we then have: 
\[
\hvol_{(X,D)}(\ord_{S_\epsilon})=\hvol_{(C_0,E_0)}(\ord_{S_\epsilon})=\hvol_{(C,C_D)}(\ord_{S_\epsilon}).
\]
In the rational coweight cone $N_{\mathbb{Q}}\cong \mathbb{Q}^2$, the one parameter subgroup $\lambda_0(t)$, which degenerates $(X, D)$ to $(C, C_D)$, corresponds to the coweight vector $(1,0)$ and 
the one parameter subgroup $\lambda'(t)$, which degenerates $(C,C_D)$ to $(C_0,E_0)$, corresponds to the coweight vector $(0,1)$. 
By construction $\pi_{\epsilon}$ corresponds to the coweight $(1,\epsilon)$ and induces the one parameter subgroup $\lambda_\epsilon: \bC^*\rightarrow {\rm GL}(p,\bC)$ which preserves $(C_0, E_0)$. 

Consider the valuations $\wt_{\lambda_t}\in \Val_{C_0, o_{C_0}}$ induced by the coweight vector of the form $(1, t)\in N_{\mathbb{R}}$ for any $t\in [0, \infty)$. Note that $\wt_{\lambda_t}$ is just the valuation associated to the Reeb vector field on $C_0$ that generates $\lambda_t$ (see \cite{LX17}). 
Define the function $f(t)=\hvol_{(C_0,E_0)}(\wt_{\lambda_t})$. Then we know that $f(t)$ is a smooth convex function on $[0,+\infty)$ (see \cite[C.2]{MSY08} and \cite[Proposition 2.21]{LX17}). 
It satisfies that
$$f(\epsilon)=\hvol_{(C_0,E_0)}(\ord_{S_\epsilon})=\hvol_{(X,D)}(\ord_{S_\epsilon})\ge \hvol_{(X,D)}(\ord_{S})=\hvol_{(C_0, E_0)}(\ord_{S_0})=f(0).$$
The above inequality is because $\ord_S$ is assumed to be a minimizer of $\hvol_{(X,D),o}$.
Using the convexity, this implies that $f(t)$ is an increasing function on $t$. Recall that the coweight $(0,1)=\lim_{t\rightarrow+\infty}t^{-1}(1,t)$ corresponds to the Koll\'ar component $F_0$ which is the degeneration of $F$ over $C_0$. By the rescaling invariance of $\hvol$, we see that we have  $\lim_{t\to +\infty}f(t)=\hvol_{(C_0,E_0)}(\ord_{F_0})$ (cf. Remark \ref{p-degvol} or \cite[Proof of Theorem 3.5]{LX17}). So we indeed have :
\[
\hvol_{(C, C_D)}(\ord_F)=\hvol_{(C_0,E_0)}(\ord_{F_0})\ge \hvol_{(C_0,E_0)}(\ord_{S_0})= \hvol_{(C,C_D)}(\ord_{S_{0}}) .
\]

\bigskip

It remains to verify the claim on $S_\epsilon$. For that
we define a filtration:
\begin{eqnarray*}
\cF^N R&=&{\rm Span}_{\bC}\left\{F_1^{a_1}\dots F_p^{a_p};\  \pi_\epsilon  \left( \sum_{i=1}^p a_i (m_i, u_i) \right) \ge N\right\}\\
&=& \{ g\in R; \text{ there exists } G\in P \text{ such that } G|_X=g \text{ and } \deg_{\pi_\epsilon} (G)\ge N  \}.
\end{eqnarray*}
Then $\{\cF^N R\}$ is the filtration induced by the weighted blow up $\widehat{\bC^p}\rightarrow \bC^p$. The associated graded ring of $\{\cF^N R\}$ is isomorphic to $\gr_w R$ with the grading given by the weight function $\pi_{\epsilon}\circ w$. Because $\gr_wR$ is a normal integral domain, by Lemma \ref{lem-fil2val}, the above filtration is induced by a valuation $w_\epsilon$ on $R$, which is a divisorial valuation. Indeed, denote the strict transform of $X$ under the weighted blow-up (i.e., filtered blow up) $\widehat{\bC^p}\rightarrow \bC^p$ by $\hat{X}$. Then, by the discussion in Section \ref{sec-filtration}, the exceptional divisor $\hat{X}\rightarrow X$ is isomorphic to $S_\epsilon=C_0/\lambda_\epsilon:=(C_0\setminus \{o_{C_0}\})/\bC^*$ and $w_\epsilon=c\cdot \ord_{S_\epsilon}$ for some $c>0$. By Proposition \ref{prop-Kolseq},
 $$(S_\epsilon, \Delta_{\epsilon})=(C_0, E_0) / \lambda_\epsilon:=(C_0\setminus \{o_{C_0}\},  E_0\setminus \{o_{C_0}\})/\bC^*$$
is indeed a klt log Fano pair and a Koll\'{a}r component over $o\in (X, D)$.  

\end{proof}

\begin{prop}\label{prop-Kolseq}
With the above notations, for any $0<\epsilon\ll 1$ with $\epsilon\in \mathbb{Q}_{+}$, let $(S_\epsilon, \Delta_\epsilon)=(C_0, E_0)/\lambda_\epsilon$. Then $S_\epsilon$ is a Koll\'{a}r component over  $o\in (X,D)$ and $o_C\in (C,C_D)$.
\end{prop}
\begin{proof}
For $0<\epsilon\ll 1$ with $\epsilon\in \mathbb{Q}_{+}$, $\lambda_\epsilon$ is associated to a $\bC^*$-action. We have a log orbifold $\bC^*$-bundle
$\pi: (C_0^{\circ}, E_0^{\circ}):=(C_0\setminus \{o_{C_0}\},  E_0\setminus \{o_{C_0}\}) \rightarrow (S_\epsilon, \Delta_\epsilon)$. The Chern class of this orbifold $\bC^*$-bundle, denoted by $c_1(C_0^{\circ} /S_\epsilon)$, is contained in ${\rm Pic}(S_\epsilon)$ and is ample. One can extract $S_\epsilon$ over $C_0$ to get a birational morphism $\mu: Y_\epsilon\rightarrow C_0$ with the exceptional divisor isomorphic to $S_\epsilon$. (We note that this is an example of the Dolgachev-Pinkham-Demazure construction, see e.g. \cite{Kol04}.)

Because $C_0$ has a $\bQ$-Gorenstein klt singularity at $o_{C_0}$, by \cite[40-42]{Kol04} we know that $c_1(C_0^{\circ}/S_\epsilon)=-r^{-1} (K_{S_\epsilon}+\Delta_\epsilon)$ for $r\in \bQ_{>0}$ and $(S_\epsilon, \Delta_\epsilon)$ has klt singularities. 
So $(S_\epsilon, \Delta_\epsilon)$ is a Koll\'{a}r component over $v\in (C_0,  E_0)$. 

To transfer this to $(X, o)$, we notice that by our construction the graded ring $\gr_{w_\epsilon}R$ is isomorphic to 
$${\rm gr}_{\wt_{\lambda_\epsilon}} \bC[C_0]={\rm gr}_{\wt_{\lambda_\epsilon}} (\gr_wR)(\cong \gr_wR).$$ The exceptional divisor of the filtered blow-up over $X$ associated to $w_\epsilon$ is isomorphic to that associated to $\wt_{\lambda_{\epsilon}}$ over $C_0$, which is ${\rm Proj}({\rm gr}_{\wt_{\lambda_\epsilon}}(\gr_wR))$ and isomorphic to $(S_\epsilon, \Delta_\epsilon)$. Since $(S_\epsilon, \Delta_\epsilon)$ is klt, by the inversion of adjunction we know that the filtered blow up is indeed a plt blow up and hence $S_\epsilon$ is a Koll\'{a}r component over $(X, D, o)$.  

The same argument also applies to $(C,C_D)$.
\end{proof}

\section{Examples}\label{s-exam}

In this section, we find out the minimizer for some examples of klt singularities $(X,o)=(\Spec R,\fm)$. We note that by Proposition \ref{p-inf} and Theorem \ref{p-equality}, this also explicit calculates the sharp lower bound of normalized multiplicities, i.e., 
$$\inf_{\fa} \lct(X,\fa)^n\cdot \mult(\fa)$$ for all $\fm$-primary ideals $\fa$ and gives the equality condition, which generalizes the results in \cite{dFEM04} on a smooth point. 
\begin{exmp}\label{e-example}
We explicitly compute the minimizer for quotient, $A_k$, $E_k$ and weakly exceptional singularities in the below. 
\end{exmp}
\begin{enumerate}
\item  Let $(X,o)=(\mathbb{C}^n,0)/G$ be an $n$-dimensional quotient singularity. Let $E\cong \mathbb{P}^{n-1}$ be the exceptional divisor over $\mathbb{C}^n$ obtained by blowing up $0$. Then denote by $S$ the valuation over $(X,o)$ which is the quotient of $E$ by $G$. Applying Lemma \ref{l-finite} to the pull back of Koll\'ar components on $X$, we know that 
$$\hvol_{X,o}(\ord_S)\le \hvol_{X,o}(\ord_F)$$ for any Koll\'ar component $F$ over $(X,o)$. So $\hvol_{X,o}$ minimizes at $\ord_S$ with 
$$\vol(o, X)=\hvol_{X,o}(\ord_S)=\frac{n^n}{|G|}.$$ For $n=2$, this is proved in \cite[Example 4.9]{LL16}.

\item
Consider the $n$-dimensional $A_{k-1}$ singularity:
$$
X=A^{n}_{k-1}:=\{z_1^2+\cdots+z_n^2+z_{n+1}^k=0\}.
$$
We consider cases when $k>\frac{2(n-1)}{n-2}$ (for other cases, see \cite[Example 4.7]{LL16}). We want to show that the valuation corresponding to the weight
$w_*=(n-1, \cdots, n-1, n-2)$ is a minimizer among all valuations in $\hvol_{(X,D),o}$. In \cite[Example 2.8]{Li15a}, these are computed out as the minimizer among all valuations obtained by weighted blow ups on the ambient space $\mathbb{C}^{n+1}$. 

We notice that under the weighted blow up corresponding to $w_*$, we have a 
birational morphism $Y\rightarrow X$ with exceptional divisor $S$ isomorphic to the weighted hypersurface
$$
S:=\{Z_1^2+\cdots+Z_n^2=0\}\subset \bP(n-1, \cdots, n-1, n-2)=:\bP_{w_*}.
$$
Because $\bP_{w_*}\cong \bP(1,\cdots, 1, n-2)$, it is easy to see that $S$ is isomorphic to $\bar{C}(Q, -K_Q)$ where 
$Q=Q^{n-2}=\{Z_1^2+\cdots+Z_n^2=0\}\subset\bP^{n-1}$ (notice that $K_{Q}^{-1}=(n-2)H$). On the other hand, because $\bP_{w_*}$ is not well-formed, we have
codimensional one orbifold locus along the infinity divisor $Q_\infty\subset S$ with the isotropy group $\bZ/(n-1)\bZ$. So the corresponding Koll\'{a}r component
is the log Fano pair $\left(S, (1-\frac{1}{n-1})Q_\infty\right)$. Because $Q_\infty$ has KE, by \cite{LL16} there is a conical KE on the pair $\left(S, (1-\frac{1}{n-1})Q_\infty\right)$.
So by Theorem \ref{t-main} and Theorem \ref{t-main2}, $\ord_S$ is indeed a global minimizer of $\hvol$ that is the unique minimizer among all Koll\'{a}r components.
Notice that for any higher order klt perturbation of these singularities, $w_*$ is also a minimizer.

\item 
We can also use Theorem \ref{t-main} to verify that the valuations in \cite[Example 2.8]{Li15a} for $E_k$ (k=6,7,8) are indeed minimizers in $\hvol_{(X,D),o}$, which are unique among Koll\'{a}r components. To avoid repetition, we will only do this for $E_7$ 
singularities. The argument for other two cases are similar. So consider the $(n+1)$-dimensional $E_7$ singularity:
\[
X^{n+1}=\{z_1^2+z_2^2+\cdots+z_n^2+z_{n+1}^3z_{n+2}+z^3_{n+2}=0\}\subset\bC^{n+2}.
\]
\begin{enumerate}
\item If $n+1=2$, then $X^2$ is a quotient singularity $\bC^2/E_7$ and so we get the unique polystable component by \cite[Example 4.9]{LL16} and example 1 above.

\item If $n+1=3$, then $X^3=\{z_1^2+z_2^2+z_3^3z_4+z_4^3=0\}\subset \bC^4\cong \{w_1 w_2+w_3^3w_4+w_4^3=0\}\subset \bC^4$ by the change of variables. This singularity has a 
$(\bC^*)^2$-action and is an example of $T$-variety of complexity one. By the recent work in \cite[Theorem 7.1 (II)]{CoSz16}, $X^3$ indeed has a Ricci flat cone K\"{a}hler metric 
associated to the canonical $\bC^*$-action associated to $w_*$. So by \cite[Theorem 1.7]{LL16}, the unique K-polystable Koll\'{a}r component is given by the orbifold $X^3/\langle w_*\rangle$.

\item $n+1=4$, then under the weighted blow up corresponding to $w_*=(9,9,9,5,6)$, we have a birational morphism 
$\hat{X}\rightarrow X$ with exceptional divisor $E$ isomorphic to the weighted hypersurface
$$
E=\{z_1^2+z_2^2+z_3^2+z_5^3=0\}\subset \bP(9, 9, 9, 5, 6)=\bP(w_*).
$$
Since $\bP(w_*)$ is not well-formed, we have:
$$
E\cong \{z_1^2+z_2^2+z_3^2+z_5^3=0\}\subset \bP(3, 3, 3, 5, 2)=\bP'. 
$$
with orbifold locus of isotropy group $\bZ/3\bZ$ along $$
V=\{z_1^2+z_2^2+z_3^2+z_5^3=0\}\subset \bP(3,3,3,2).
$$

Alternatively, $E$ is a weighted projective cone over the weighted hypersurface. It is easy to see that as an orbifold $(V, \Delta) \cong \left(\bP^2, (1-\frac{1}{3})Q\right)$ where $Q=\{z_1^2+z_2^2+z_3^2=0\}\subset \bP^2$.
By \cite{LS14}, there exists an orbifold K\"{a}hler-Einstein metric on $(V, \Delta)$.
Notice that $-(K_V+\Delta)=3L-\frac{4}{3}L=\frac{5}{3}L$ where $L$ is the hyperplane bundle of $\bP^2$. Denoting by $H$ the hyperplane bundle of $\bP'$, then
$H|_V=L/3$. If $V$ is considered as a divisor of $E$, then 
$$V|_V=\big(\{z_4=0\}\cap E\big)=5H|_V=\frac{5}{3}L.$$ So $-(K_V+\Delta)= V|_V$. 

Then by \cite[Theorem 1.7]{LL16}, there exists an orbifold K\"{a}hler-Einstein metric on $E$ because the cone angle at infinity is $\beta=1/3$. Thus the unique log-K-semistable (actually log-K-polystable) Koll\'{a}r component is given by the pair $\left(E, \left(1-\frac{1}{3}\right)V\right)$.

\item
$n+1=5$, under the weighted blow up corresponding to $w_*=(3, 3, 3, 3, 2, 2)$, we have a birational morphism $\hat{X}\rightarrow X$ with exceptional divisor $E$ isomorphic to the 
weighted hypersurface:
\[
E=\{z_1^2+z_2^2+z_3^2+z_4^2+z_6^3=0\}\subset \bP(3, 3, 3, 3, 2, 2)=:\bP(w_*).
\]
This is a weighted projective cone over the weighted hypersurface:
\[
V=\{z_1^2+z_2^2+z_3^2+z_4^2+z_6^3=0\}\subset \bP(3, 3, 3, 3, 2).
\]
As orbifold, we have $(V, \Delta)=\left(\bP^3, (1-\frac{1}{3})Q\right)$. By \cite{LS14, Li13}, $(V, \Delta)$ is log-K-semistable and degenerates to a conical K\"{a}hler-Einstein pair. So by \cite{LL16}, we know 
that $(E, (1-\beta)V_\infty)$ is log-K-semistable. To determine $\beta$, we notice that 
$$-(K_V+\Delta)=4L-\frac{4}{3}L=\frac{8}{3}L=4\cdot \frac{2}{3}L=4\cdot V_\infty|_V.$$ 
So $\beta=1$ and we conclude that the unique (strictly) K-semistable Koll\'{a}r component is indeed the $\bQ$-Fano variety $E$.

\item $n+1\ge 6$. Under the weighted blow up corresponding to $w_*=(n-1, \dots, n-1, n-2, n-2)$, we have a birational morphism $\hat{X}\rightarrow X$ with exceptional divisor $E$ isomorphic to
the weighted hypersurface:
\[
E=\{z_1^2+\cdots+z_n^2=0\}\subset \bP(n-1, \cdots, n-1, n-2, n-2)=:\bP(w_*).
\]
This is the weighted projective cone over 
\[
V=\{z_1^2+\cdots+z_n^2=0\}\subset \bP(n-1, \cdots, n-1, n-2).
\]
By the discussion in the above $A^n_{k-1}$ singularity case, we know that as an orbifold,
$(V, \Delta)=\left(\bar{C}(Q, -K_Q), (1-\frac{1}{n-1})Q_\infty\right)$, which has an orbifold K\"{a}hler-Einstein metric. Notice that
\[
-(K_V+\Delta)=(n(n-1)+n-2)H|_V-2(n-1)H|_V=n (n-2)H|_V.
\]
By \cite[Theorem 1.7]{LL16}, the $\bQ$-Fano variety $E$ indeed has an orbifold K\"{a}hler-Einstein metric ($\beta=n/n=1$ at infinity) and hence by Theorem \ref{t-main} is the unique K-semistable (actually K-polystable) Koll\'{a}r component.

\end{enumerate}

We remark that, in the case of $D_{k+1}$ singularities, since the valuations computed out in \cite[Example 2.8]{Li15a} could be irrational, the result in this paper does not directly tell whether it is a minimizer in $\Val_{X,o}$. This irregular situation is studied in the following paper \cite{LX17} (see also \cite[section 6]{LL16}).

\item A notion called weakly-exceptional singularity is introduced in \cite{Pro00}. As the name suggested, this is a weaker notion than the exceptional singularity introduced in \cite{Sho00}, which forms a special class of singularities in the theory of local complements. In our language, a singularity $(X,o)$ is {\it weakly-exceptional} if and only if it has a unique Koll\'ar component $S$. We know that if a singularity is weakly-exceptional, then the log $\alpha$-invariant for the log Fano $(S,\Delta_S)$ is at least 1 (see \cite[Theorem 4.3]{Pro00} and \cite{CS14}). In particular, we know that $(S,\Delta_S)$ is K-semistable (see \cite[Theorem 1.4]{OS14} or \cite[Theorem 3.12]{Ber13}). And by Theorem \ref{t-main} and \ref{t-main2}, we know $\ord_S$ is the unique minimizer of $\hvol(S)$ among all Koll\'ar components. See \cite{CS14} for examples of weakly exceptional singularities. 
\end{enumerate}

\bigskip

Finally, we should point out that there are examples of minimizers from Sasakian-Einstein metrics. See \cite{LL16, LX17} for details. The works in \cite{LX17, LWX18} also apply minimization of normalized volumes to Donaldson-Sun's conjecture about metric tangent cones on Gromov-Hausdorff limits of K\"{a}hler-Einstein Fano manifolds.



\noindent Chi Li

\noindent{Department of Mathematics, Purdue University, West Lafayette, IN 47907-2067}

\noindent{li2285@purdue.edu}

\bigskip

\noindent Chenyang Xu

\noindent {Beijing International Center for Mathematical Research,
       Beijing 100871, China}

\noindent    {cyxu@math.pku.edu.cn}

\medskip

\noindent {{\it Current address}: MIT, Cambridge, MA 02139, USA}

\noindent    {cyxu@math.mit.edu}

\end{document}